\documentclass[11pt]{amsart}
\usepackage{amsmath,amsfonts,amssymb,mathabx}
\usepackage{amsthm}
\usepackage{amstext}
\usepackage[normalem]{ulem}
\usepackage{verbatim}
\usepackage{longtable}

\usepackage{fancyhdr} 
\usepackage{graphicx}  
  \usepackage{bbm}
\usepackage{geometry}
\usepackage{enumitem}
\usepackage{hyperref}
\hypersetup{
	colorlinks,
	linkcolor={blue},
	citecolor={red},
	urlcolor={blue}
}
\usepackage{cancel}
\makeatletter
\newcommand{\setword}[2]{ 
	\phantomsection
	#1\def\@currentlabel{\unexpanded{#1}}\label{#2} 
}
\makeatother
\usepackage[usenames,dvipsnames]{color}

   \definecolor{spoc}{RGB}{180, 90, 20}
 \definecolor{spic}{RGB}{27, 121, 0}

\renewcommand{\labelenumi}{(\roman{enumi})}

\newtheorem{thm}{Theorem}[section]
\newtheorem{lem}[thm]{Lemma}

\newtheorem{rem}[thm]{Remark}
\newtheorem{rems}[thm]{Remarks}
\newtheorem{prop}[thm]{Proposition}
\newtheorem{df}[thm]{Definition}
\newtheorem{dfs}[thm]{Definitions}
\newtheorem{ex}[thm]{Example}
\newtheorem{cex}[thm]{Counter-example}
\newtheorem{exs}[thm]{Examples}
\newtheorem{symbs}[thm]{Notations}

\newcommand{\1}{\mathbbm{1}}

\newcommand{\wt}{\widetilde}

\newcommand{\wb}{\widebar}
\newcommand{\vS}{\varSigma}
 \newcommand{\vO}{\varOmega}
\newcommand{\vT}{\varTheta}
\newcommand{\ov}{\overline}
\newcommand{\mf}{\mathfrak}

\newcommand{\uph}{\upharpoonright}
\newcommand{\restr}{\upharpoonright}
\newcommand{\s}{\sigma}
\newcommand{\sq}{\subseteq}

\def\Q{{\mathbb Q}}
\newcommand{\B}{\mathfrak B}
\newcommand{\N}{\mathbb N}

\newcommand{\R}{\mathbb R}
\newcommand{\T}{\mathbb T}
\newcommand{\E}{\mathbb E}

\newcommand{\A}{\mathcal A}
\newcommand{\F}{\mathcal F}

\newcommand{\G}{\mathcal G}

\newcommand{\M}{\mathcal M}
\def\H{{\mathcal H}}
\def\L{{\mathcal L}}
 
\newcommand{\al}{\alpha}
 \newcommand{\ga}{\gamma}
\newcommand{\be}{\beta}
\newcommand{\ka}{\kappa}
\newcommand{\ttheta}{\rho(\theta)}
\newcommand{\Ttheta}{\rho(\vT)}
\newcommand{\Qtd}{\{Q_{\theta}\}_{\theta\in D}}

\title[A characterization of  PEMM converting a CMRP into a CMPP with applications in RT]{A characterization of equivalent martingale  measures converting a compound mixed renewal process into a compound mixed Poisson one  with applications in Risk Theory}
\author[S.~M. Tzaninis]{Spyridon M. Tzaninis}
\address{Department of Statistics and Insurance Science\\ University of
	Piraeus\\ 80 Karaoli and Dimitriou str.\\ 185 34 Piraeus\\ Greece}
\email{stzaninis@unipi.gr; macheras@unipi.gr}
\thanks{}

\author[N.~D. Macheras]{Nikolaos D. Macheras}

\date{\today}

\begin{document}
 
\begin{abstract}
 Given  a compound mixed renewal process $S$  under a probability measure $P$, we provide a characterization of all  equivalent martingale probability measures $Q$ on the domain of $P$, that convert $S$ into a compound mixed Poisson process.  This result  extends earlier works of Delbaen \& Haezendonck \cite{dh}, Lyberopoulos \& Macheras \cite{lm3} and  the authors \cite{mt3}, and enables us to find a wide class of price processes satisfying the condition of no free lunch with vanishing risk. Implications to the ruin problem and to the computation of premium calculation principles in an arbitrage free insurance market are also discussed.
 
\end{abstract}
\maketitle

\section{Introduction}\label{intro}
Given a price process $U_\mathbb{T}:=\{U_t\}_{t\in\mathbb{T}}$, where $ \mathbb{T}:=[0,T]$, $T>0$,  on a probability space $(\vO,\vS,P)$, a basic method in Mathematical Finance is to replace the initial probability measure $P$ by an equivalent one $Q$, which converts  $U_\mathbb{T}$  into a martingale with respect to $Q$. The new probability measure, often called {\em risk-neutral} or {\em equivalent martingale measure} (written EMM for short), is then used for pricing and hedging contingent claims (e.g., options, futures, etc.).  Note that such a method for pricing contingent claims is originated from the field of Actuarial Science (see Delbaen \& Schachermayer \cite{ds}, pages 149-150 for more details).   However, in contrast to the situation of the classical Black-Scholes option pricing formula, where the EMM is unique, in Actuarial Mathematics that  is certainly not the case, as the insurance market is not, in general, complete (see e.g., Sondermann \cite{so}, Section 4). Thus, if $U_\mathbb{T}$ represents the liabilities of an insurance company, then there exist infinitely many equivalent martingale measures for $U_\mathbb{T}$, so that pricing is directly linked with an attitude towards risk, see \cite{dh},  pages 269-270, for more details. The latter led Delbaen \& Haezendonck \cite{dh} to the problem of characterizing all those EMMs $Q$ which preserve the structure of a given compound Poisson process under $P$, see \cite{dh}, Proposition 2.2.  The work of Delbaen \& Haezendonck played a key role in  understanding the interplay between financial and actuarial pricing of insurance (see  Embrechts \cite{em} for an overview), and has influenced the studies of many researchers (see  \cite{mt3},  page 44, and the references therein for more details).   Nevertheless,  such a characterization of  EMMs for $U_\mathbb{T}$ does not always provide a viable pricing system in  actuarial practice, since it is not appropriate for describing inhomogeneous risk portfolios. For this reason the work of Delbaen \& Haezendonck \cite{dh} was generalized by Embrechts \& Meister \cite{emme} and Lyberopoulos \& Macheras \cite{lm3, lm3aap} to mixed Poisson risk models.  However, since the (mixed) Poisson risk model presents some serious deficiencies as far as practical models are considered (see \cite{mt3},  page 44, and \cite{mt2}, page 226, and the references therein),   it seems reasonable to investigate the existence of EMMs for the price process $U_\mathbb{T}$ in the more   general mixed renewal risk model.  \smallskip

In \cite{mt2}, Corollary  4.8, a characterization of all progressively equivalent probability measures that convert a  compound MRP into a compound mixed Poisson one (written MPP for short) was proven.  In Section \ref{CRPM}, relying on the above result, we provide   a characterization of all progressively EMMs $Q$ for a canonical price process that convert a  compound MRP under $P$ into a  compound  MPP under $Q$, see Theorem \ref{class}. This theorem generalizes  corresponding results of \cite{lm3},  Proposition  5.1(ii), and \cite{mt3},  Proposition 4.2.   A first consequence of Theorem \ref{class} is Theorem \ref{ftap}, where we find out a wide class of \textit{canonical} price processes satisfying the condition of no free lunch with vanishing risk (written (NFLVR) for short), connecting in this way our results with this basic notion of Mathematical Finance.
\smallskip

An implication of Theorem \ref{class} to the ruin problem is discussed in Section \ref{rp}, where for a given reserve process  $R^u(\vT)$ we characterize all those progressively equivalent to $P$ probability measures $Q$  that covert  a $P$-compound MRP into a $Q$-compound MPP in such a way that ruin occurs $Q$-a.s., see Theorem \ref{ruin1a}.   In Section \ref{ap}, we discuss some implications of our results to the pricing of actuarial risks (premium calculation principles)  in an arbitrage free insurance market.   Finally, in Section \ref{ex} we present some  concrete  examples demonstrating how to construct mixed premium calculation principles (see Section \ref{ap} for the definition) in   an insurance market possessing the property of (NFLVR) and how  to obtain   explicit formulas for  their corresponding ruin   probabilities.

\section{Preliminaries}\label{prm}
$\N$, $\Q$ and $\R$ stand for the natural, the  rational and the real numbers, respectively, while  $\N_0:=\N\cup\{0\}$, $\ov{\N}_0:=\N_0\cup\{\infty\}$, $\ov\R:=\R\cup\{-\infty,+\infty\}$, $\Q_+:=\{x\in\Q: x\geq{0}\}$ and  $\R_+:=\{x\in\R: x\geq{0}\}$. If $d\in\N$ then $\R^d$ denotes the Euclidean space of dimension $d$. For a map $f:A\rightarrow{E}$ and for a non-empty set $B\sq{A}$ write $f{\restr}{B}$ for the restriction of $f$ to $B$  and $\1_B$ for the indicator function of the set $B$.\smallskip

{\em Throughout this paper, unless stated otherwise, $(\vO,\vS,P)$ is a fixed but arbitrary probability space and $\vT:\vO\to D\subseteq\R^d$ ($d\in\N$) is a $d$-dimensional random vector}.  By  $\mathcal{L}^\ell(P)$ we denote the family of all $\vS$-measurable real-valued  functions $f$ on $\vO$ such that $\int|f|^\ell\;dP<\infty$ ($\ell\in\{1,2\}$).  For any Hausdorff topology $\mathfrak{T}$ over $\vO$, by $\B(\vO)$ is denoted the Borel $\s$-algebra on $\vO$, i.e., the $\s$-algebra generated by $\mathfrak{T}$, while $\mf{B}:=\mf{B}(\R)$  and  $\B_d:=\B(\R^d)$, where $d\in\N$, stand for the Borel $\s$-algebras of subsets of $\R$ and $\R^d$, respectively, generated by the Euclidean topology $\mathfrak{E}$ over $\R$ and by the product topology $\mathfrak{E}^d$ over $\R^d$, respectively. Put $\ov \B:=\B(\ov\R):=\{A\subseteq\ov\R : A\cap\R\in\B\}$ and $\B(J):=\{B\cap J : B\in\ov\B\}$ for every interval $J$ in $\ov\R$. In particular, write $\B(0,\infty):=\B\big((0,\infty)\big)$, for simplicity. Our measure theoretic terminology is standard and generally follows  \cite{Co}.    For the definitions of real-valued random variables and random variables we refer to  \cite{Co},  page 308.   We apply   notation $P_{X}:=P_X(\theta):={\bf{K}}(\theta)$ to mean that $X$ is distributed according to the law ${\bf{K}}(\theta)$, where $\theta\in D\subseteq\R^d$  is the parameter of the distribution.  Notation  ${\bf Ga}(b,a)$, where $a,b\in(0,\infty)$, stands for the law  of gamma  distribution (cf., e.g., \cite{Sc},  page 180). In particular, ${\bf Ga}(b,1)={\bf Exp}(b)$ stands for the law of exponential distribution.  For the unexplained terminology of Probability and Risk Theory we refer to \cite{Sc}. \smallskip

Given a random variable $X$,  a {\bf conditional distribution of $X$ over $\vT$} is a $\s(\vT)$-$\mathfrak{B}$-Markov kernel (see \cite{ba},  Definition 36.1 for the definition) denoted by $P_{X\mid\vT}:=P_{X\mid\sigma(\vT)}$ and satisfying for each $B\in\mf{B}$ the equality
$P_{X\mid\vT}(\bullet,B)=P\big(X^{-1}(B)\mid\sigma(\vT)\big)(\bullet)$ ${P}{\restr}\sigma(\vT)$-almost surely (written a.s. for short). Clearly, for every $\mathfrak{B}(\R^d)$-$\mathfrak{B}$-Markov kernel $k$, the map $K(\vT)$ from $\vO\times\mathfrak{B}$ into $[0,1]$ defined by means of
\[
K(\vT)(\omega,B):=\big(k(\bullet,B)\circ\vT\big)(\omega)
\quad\mbox{for any}\;\;(\omega,B)\in \vO\times\mathfrak{B}
\]
is a $\s(\vT)$-$\mathfrak{B}$-Markov kernel. Then for $\theta=\vT(\omega)$ with $\omega\in\vO$ the probability measures $k(\theta,\bullet)$ are distributions on $\mathfrak{B}$ and so we may write $\mathbf{K}(\theta)(\bullet)$ instead of $k(\theta,\bullet)$. Consequently, in this case $K(\vT)$ will be denoted by 
$\mathbf{K}(\vT)$.\smallskip

For any real-valued random variables $X$, $Y$ on $\vO$ we say that $P_{X|\vT}$ and $P_{Y|\vT}$ are $P{\uph}\sigma(\vT)$-equivalent and we write $P_{X|\vT}=P_{Y|\vT}$ $P{\uph}\sigma(\vT)$-a.s.,  if there exists a $P$-null set $M\in\sigma(\vT)$ such that for any $\omega\notin M$ and $B\in \B$ the equality $P_{X|\vT}(\omega,B)=P_{Y|\vT}(\omega,B)$ holds true.\smallskip

 A sequence $\{V_n\}_{n\in \N}$ of real-valued random variables on  $\vO$ is said  to be:\smallskip 
 
 \noindent $\bullet$ $P$-{\bf conditionally (stochastically) independent over $\sigma(\vT)$} if, for each $n\in\N$ with $n\geq 2$ we have
	\[
	P\Big(\bigcap^{n}_{j=1} \{V_{i_j}\leq v_{i_j}\}\mid\sigma(\vT)\Big)=\prod^{n}_{j=1} P\big(\{V_{i_j}\leq v_{i_j}\}\mid\sigma(\vT)\big)\quad P{\uph}\sigma(\vT)\mbox{-a.s.},
	\]
whenever $i_1,\ldots, i_n$ are distinct members of $I\subseteq\N$ and $(v_{i_1},\ldots v_{i_n})\in\R^n$;  \smallskip
	
 \noindent $\bullet$  {\bf $P$-conditionally identically distributed over $\sigma(\vT)$} if, 
	\[
	P\big(F\cap V_k^{-1}[B]\big)=P\big(F\cap V_m^{-1}[B]\big)
	\]
	whenever $k,m\in \N$, $F\in\sigma(\vT)$ and $B\in \B$. \smallskip
	
	\noindent  We say that the process $\{V_n\}_{n\in \N}$ is $P${\bf-conditionally (stochastically) independent or identically distributed given $\vT$}, if it is conditionally independent or identically distributed over the $\sigma$-algebra  $\sigma(\vT)$.  \smallskip

{\em Throughout what follows  we write  ``conditionally" in the place of ``conditionally given $\vT$” whenever conditioning refers to  $\vT$.}\smallskip

For the definitions of a \textit{counting} (or \textit{claim number}) \textit{process} $N:=\{N_t\}_{t\in\R_+}$, an \textit{arrival process} $T:=\{T_n\}_{n\in\N_0}$ induced by $N$, an \textit{interarrival process} $W:=\{W\}_{n\in\N}$ induced by $T$, a \textit{claim size process} $X:=\{X_n\}_{n\in\N}$ and an \textit{aggregate  claims process} $S:=\{S_t\}_{t\in\R_+}$ induced by $N$ and $X$, we refer to \cite{Sc}. We assume that $P(\{X_n>0\})=1$ for all $n\in\N$. Recall that a pair $(N,X)$ is called a {\bf risk process}, if $N$ is a counting process, $X$ is $P$-i.i.d. and the processes $N$ and $X$ are $P$-mutually independent (see \cite{Sc},  page 127).  \smallskip

Recall that a counting process $N$  is said to be a $P$-{\bf mixed renewal process with mixing parameter $\vT$ and  interarrival time conditional distribution $\bf{K}(\vT)$} (written $P$-MRP$(\bf{K}(\vT))$ for short), if the  interarrival process $W$ is $P$-conditionally independent and for all $n\in\N$ condition
$P_{W_n\mid \vT}={\bf K}(\vT)$ $P{\restr}\s(\vT)$-a.s. is valid (see  \cite{lm6z3},  Definition 3.1).  In particular, if the distribution of $\vT$ is degenerate at some point $\theta_0\in D$, then the counting process $N$ becomes  a $P$-{\em renewal process with  interarrival time distribution ${\bf K}(\theta_0)$} (written $P$-RP$({\bf K}(\theta_0))$ for short).  If $N$ is a $P$-MRP$(\bf{K}(\vT))$ then according to \cite{mt2}, Corollary 3.5(ii), it has zero probability of explosion i.e., $P\big(\{\sup_{n\in\N_0} T_n<\infty\}\big)=0$.\smallskip

Accordingly,  an aggregate claims process $S$ induced by a $P$-risk process $(N,X)$  such that $N$ is a $P$-MRP$(\bf{K}(\vT))$ is called  a  {\bf compound mixed renewal process  with parameters ${\bf K}(\vT)$ and $P_{X_1}$} (written $P$-CMRP$({\bf K}(\vT),P_{X_1})$ for short). In particular, if the distribution of $\vT$ is degenerate at $\theta_0\in D$ then $S$ is called a  {\em compound  renewal process with parameters ${\bf K}(\theta_0)$ and $P_{X_1}$} ($P$-CRP$({\bf K}(\theta_0),P_{X_1})$ for short).\smallskip

{\em Throughout what follows we denote again by ${\bf K}(\vT)$ and ${\bf K}(\theta)$ the conditional distribution function and the distribution function  induced by the conditional probability distribution ${\bf K}(\vT)$ and the  probability distribution ${\bf K}(\theta)$, respectively.} \smallskip

Since conditioning is involved in the definition of (compound) mixed renewal processes, it is expected that regular conditional probabilities will play a fundamental role in their analysis.   To this purpose, recall  the following definition  (cf., e.g., \cite{mt2},  Definition 3.2).

\begin{df}\label{rcp}
	\normalfont
Let $(Z,H,R)$ be a probability  space. A family $\{P_z\}_{z\in Z}$ of probability measures on $\vS$ is called a {\bf regular conditional probability} (rcp for short) of $P$ over $R$  if for any fixed $E\in\vS$ the map $z\mapsto P_z(E)$ is $H$-measurable, and $\int P_{z}(E)\,R(dz)=P(E)$ for every $E\in \vS$.  If $f:\vO\rightarrow Z$ is an inverse-measure-preserving function (i.e., $P(f^{-1}(B))=R(B)$ for each $B\in{H}$), a rcp $\{P_{z}\}_{z\in Z}$ of $P$ over $R$ is called {\bf consistent} with $f$ if, for each $B\in{H}$, the equality $P_{z}(f^{-1}(B))=1$ holds for $R$-almost every $z\in B$. 
\end{df}

We say that a rcp $\{P_{z}\}_{z\in Z}$ of $P$ over $R$ consistent with $f$ is {\bf  essentially unique}, if for any other rcp $\{\wt P_{z}\}_{z\in Z}$ of $P$ over $R$ consistent with $f$ there exists a $R$-null set $M\in H$ such that for any $z\notin M$ the equality $P_z=\wt P_z$ holds true.\smallskip

{\em From now on  $(Z,H,R):=(D,\B(D), P_\vT)$ and the family $\{P_\theta\}_{\theta\in D}$ is a rcp of $P$ over $P_\vT$ consistent with $\vT$.}\smallskip

Let $\T\sq\R_+$. For a process $Y_\T:=\{Y_t\}_{t\in\T}$ denote by $\F_\T^Y:=\{\F^Y_t\}_{t\in\T}$ the canonical filtration of $Y_\T$. For  $\T=\R_+$ or $\T=\N$ we simply write $\F^Y$ instead of $\F^Y_{\R_+}$ or $\F^Y_{\N}$, respectively. Also, we write $\F:=\{\F_t\}_{t\in\R_+}$, where $\F_t:=\s(\F^S_t\cup\s(\vT))$, for the canonical filtration of $S$ and $\vT$, $\F^S_\infty:=\s(\bigcup_{t\in\R_+}\F^S_t)$ and $\F_\infty:=\s(\F^S_\infty\cup\s(\vT))$.

\begin{df}\label{df0}
	\normalfont
Let  $Q$ be a probability measures on $(\vO,\vS)$.\smallskip

\noindent  \textbf{(a)} 	
If $\G\subseteq\vS$ is a $\s$-algebra of subsets of $\vO$, then  $Q$ is \textit{absolutely continuous}  with respect to $P$ on $\G$ if $P(A)=0$ implies $Q(A)=0$ for all $A\in\G$. The measures 	$Q$ and $P$ are \textit{equivalent} on $\G$ (in symbols $Q{\uph}\G\sim P{\uph}\G$) if $Q$ is absolutely continuous with respect to $P$ on $\G$ and vice versa. If $\mathcal{G}=\vS$ simply write $P\sim{Q}$.
	\smallskip
	
\noindent  	\textbf{(b)} If $\{\G_t\}_{t\in\R_+}$ is a filtration for $\vS$, then $Q$ and $P$ are \textit{progressively equivalent} with respect to $\{\G_t\}_{t\in\R_+}$, if $Q$ and $P$   are equivalent on each  $\G_t$ (i.e., $Q{\uph}\G_t\sim P{\uph}\G_t$).
\end{df}

The following conditions will be useful for our investigations:
\begin{itemize}
	\item[{\bf(a1)}] The process $W$ and $X$ are $P$-conditionally  mutually independent.
	\item[{\bf(a2)}] The random vector $\vT$ and the process $X$ are $P$-(unconditionally) independent.
\end{itemize}

{\em Next, whenever condition (a1) or (a2) holds true we shall write that the quadruplet $(P,W,X,\vT)$ or (if no confusion arises) the probability measure $P$ satisfies (a1) or (a2), respectively}.

\begin{symbs}\label{symb2} 
\normalfont 
Denote by $\hypertarget{mkd}{\mathfrak{M}^k(D)}$, $k\in\N$, the class of all $\mathfrak{B}(D)$-$\B_k$-measurable functions on $D$. In the special case $k=1$ write $\mathfrak{M}(D):=\mathfrak{M}^1(D)$ and $\mathfrak{M}_+(D)$ for the class of all positive elements of $\mathfrak M(D)$. 
 Fix on arbitrary $\ell\in\{1,2\}$ and $\rho\in \mathfrak{M}^k(D)$.\smallskip
	
\noindent {\bf (a)} The class of all real-valued $\B\big((0,\infty)\times D\big)$-measurable functions $\be$ on $(0,\infty)\times D$, defined by means of $\be(x,\theta):=\ga(x)+\al(\theta)$ for any  $(x,\theta)\in(0,\infty)\times D$, where $\al\in\hyperlink{mkd}{\mathfrak M(D)}$ and $\ga$ is a real-valued $\B(0,\infty)$-measurable function satisfying conditions  $\E_{P}\left[e^{\ga(X_{1})}\right]=1$ and $\E_{P}\left[X^\ell_1\cdot e^{\ga(X_{1})}\right]<\infty$ (resp. $\E_P\big[e^{\ga(X_1)}\big]=1$),   will be denoted by $\hypertarget{fpt}{\F^\ell_{P,\vT}}:=\F^\ell_{P,\vT,X_1}$ (resp. $\F_{P,\vT}:=\F_{P,\vT,X_1})$. 
\smallskip

\noindent {\bf(b)} The  class of all  $\xi\in\hyperlink{mkd}{\mf{M}(D)}$ such that $P_{\vT}(\{\xi>0\})=1$ and $\E_P[\xi(\vT)]=1$ is denoted by $\hypertarget{r+d}{\mathcal{R}_+(D)}:=\mathcal{R}_+(D,\mathfrak{B}(D), P_{\vT})$.  \smallskip 
	
\noindent {\bf (c)} The class of all probability measures $Q$ on $\vS$, which satisfy conditions (a1)  and  (a2),   are progressively equivalent to $P$,  and such that $S$ is a $Q$-CMRP$({\bf \Lambda}(\rho(\vT)),Q_{X_1})$,  will be denoted by $\hypertarget{ms}{\M_{S,{\bf\Lambda}(\rho(\vT))}}{:=}\M_{S,{\bf \Lambda}(\rho(\vT)),P,{X_1}}$. The class of all elements $Q$ of $\M_{S,{\bf\Lambda}(\rho(\vT))}$ with $\E_Q[X^\ell_1]<\infty$  will be denoted by $\hypertarget{msl}{\M^\ell_{S,{\bf\Lambda}(\rho(\vT))}}{:=}\M^\ell_{S,{\bf \Lambda}(\rho(\vT)),P,{X_1}}$. In the special case $d=k$ and $\rho:=id_D$ we write $\mathcal{M}_{S,{\bf\Lambda}(\vT)}:=\mathcal{M}_{S,{\bf\Lambda}(\rho(\vT))}$ and $\M^\ell_{S,{\bf \Lambda}(\vT)}:=\M^\ell_{S,{\bf \Lambda}(\Ttheta)}$  for simplicity.\smallskip

\noindent {\bf (d})  Let $\theta\in D$. Denote  by $\hypertarget{mst}{{\M}_{S,{\bf\Lambda}(\ttheta)}}$  the class of all probability measures $Q_\theta$ on $\vS$, such that $Q_\theta{\uph}\F_t\sim P_\theta{\uph}\F_t$ for any $t\geq0$ and $S$ is a $Q_\theta$-CRP$({\bf \Lambda}(\ttheta),(Q_\theta)_{X_1})$. The class of all $Q_\theta\in{\M}_{S,{\bf\Lambda}(\ttheta)}$  with $\E_{Q_\theta}[X_1^\ell]<\infty$ is denoted by $\hypertarget{mstl}{{\M}^\ell_{S,{\bf\Lambda}(\ttheta)}}$.
\end{symbs} 

 {\em From now on, unless stated otherwise, $\rho\in\mathfrak{M}^k(D)$, $k\in\N$.}

\begin{rems}\label{rem1aa} 
	\normalfont
\textbf{(a)} 
Clearly inclusions $\hyperlink{fpt}{\F^2_{P,\vT}}\sq \hyperlink{fpt}{\F^1_{P,\vT}}\sq \hyperlink{fpt}{\F_{P,\vT}}$ and $\hyperlink{msl}{\M^{2}_{S,{\bf \Lambda}(\rho(\vT))}}\sq \hyperlink{msl}{\M^{1}_{S,{\bf \Lambda}(\rho(\vT))}} \sq\hyperlink{ms}{\M_{S,{\bf \Lambda}(\rho(\vT))}}$ hold true, but simple examples show that $\hyperlink{fpt}{\F^2_{P,\vT}}\neq \hyperlink{fpt}{\F^1_{P,\vT}}\neq \hyperlink{fpt}{\F_{P,\vT}}$ and $\hyperlink{msl}{\M^{2}_{S,{\bf \Lambda}(\rho(\vT))}}\neq \hyperlink{msl}{\M^{1}_{S,{\bf \Lambda}(\rho(\vT))}} \neq\hyperlink{ms}{\M_{S,{\bf \Lambda}(\rho(\vT))}}$ in general.
\smallskip

\noindent \textbf{(b)} For $\ell\in\{1,2\}$   the following statements are equivalent: 
\begin{enumerate}
	\item $P\in\hyperlink{msl}{\M^\ell_{S,{\bf K}(\vT)}}$ with $P\big(\big\{\E_P[W_1\mid\vT]<\infty\big\}\big)=1$;
	\item  there exists a $P_\vT$-null set $W_P\in\B(D)$  such that $P_\theta\in\hyperlink{mstl}{\M^\ell_{S,{\bf K}(\theta)}}$ with $(P_\theta)_{X_1}=P_{X_1}$  and $\E_{P_\theta}[W_1]<\infty$  for any  $\theta\notin {W_P}$. 
\end{enumerate}

In fact, since $X_1$ and $\vT$ are (unconditionally) independent by (a2), we have $X_1\in\mathcal{L}^\ell(P)$ if and only if $X_1\in\mathcal{L}^\ell(P_\theta)$ for all $\theta\in{D}$, while by \cite{mt2}, Proposition 3.4, we have $P\in\mathcal{M}_{S,{\bf K}(\vT)}$ if and only if there exists a $P_\vT$-null set $L_P\in\mathfrak{B}(D)$ such that $P_\theta\in\mathcal{M}_{S,{\bf K}(\theta)}$ with $(P_\theta)_{X_1}=P_{X_1}$ for all $\theta\notin L_P$. Furthermore, by \cite{lm1v}, Lemma 3.5(i), we get that $P(\{\E_P[W_1\mid\vT]<\infty\})=1$ if and only if there exists a $P_\vT$-null set ${D_P}\in\mathfrak{B}(D)$ such that $\E_{P_\theta}[W_1]<\infty$ for all $\theta\notin {D_P}$. Putting $W_P:=L_P\cup D_P\in \B(D)$ we get the desired  equivalence of (i) and (ii).
\end{rems}

\begin{df}
\normalfont
Recall that, for given $\T\sq\R_+$  a \textbf{martingale in $\mathcal{L}^\ell(P)$ adapted to the filtration $\mathcal{Z}_\T:=\{\mathcal{Z}_t\}_{t\in\T}$}, or else a \textbf{$\mathcal{Z}_\T$-martingale in $\mathcal{L}^\ell(P)$}, is a process $Z_\T:=\{Z_t\}_{t\in\T}$ of random variables in $\mathcal{L}^\ell(P)$ such that $Z_t$ is $\mathcal{Z}_t$-measurable for each $t\in\T$, and whenever $u\leq{t}$ in $\T$ and $E\in\mathcal{Z}_u$ then $\int_E{Z}_udP=\int_E{Z}_tdP$. The latter condition is called the \textbf{martingale property} (cf., e.g., \cite{Sc},  page 25). For $\mathcal{Z}_{\R_+}=\mathcal{F}$ we simply write that $Z$ is a martingale in $\mathcal{L}^\ell(P)$. A $\mathcal{Z}_\T$-martingale $\{Z_t\}_{t\in\T}$ in $\mathcal{L}^\ell(P)$ is {\bf $P$-a.s. positive},  if $Z_t$ is $P$-a.s. positive for each $t\in\T$.
\end{df}

Given $\vO:=(0,\infty)^{\N}\times(0,\infty)^{\N}\times D$ and $\vS:=\B(\vO)=\B(0,\infty)_{\N}\otimes\B(0,\infty)_{\N}\otimes\B(D)$, let $\mu$ be a probability measure on $\B(D)$,  and let $P_{n}(\theta):={\bf{K}}(\theta)$ and $R_n:=R$ be probability measures on $\B(0,\infty)$ for any $n\in\N$ and fixed $\theta\in D$. Assume that    for any fixed $B\in\B(0,\infty)$ the function $\theta\mapsto{\bf{K}}(\theta)(B)$ is $\B(D)$-measurable. It then follows by \cite{mt2}, Proposition 4.1,  that there exist:
\begin{itemize}
\item  a family $\{P_{\theta}\}_{\theta\in D}$ of probability measures   on $\vS$ and a probability measure $P$ on $\vS$ such that $\{P_{\theta}\}_{\theta\in D}$ is a rcp of $P$ over  $\mu$ consistent with $\vT:=\pi_D$, where $\pi_D$ is the canonical projection from $\vO$ onto $D$, and $P_\vT=\mu$; 
\item  an interarrival process $W$ such that  $(P_\theta)_{W_n}=\mathbf{K}(\theta)$ for all $n\in\N$; 
\item a claim size process $X$ such that $P_{X_n}=R$  for all $n\in\N$, and  
\item a counting process $N$ and an aggregate process $S$ induced by the risk process $(N,X)$, such that $P$ is an element of $\hyperlink{ms}{\M_{S,{\bf K}(\vT)}}$.
\end{itemize}

{\em  Throughout what follows, unless stated otherwise, $(\vO,\vS,P)$, $N , W, X, S, \vT$ and $\{P_{\theta}\}_{\theta\in D}$ are as above, $\vS=\F_\infty$, $\ell\in\{1,2\}$ and  $S_t^{(\ga)}:=\sum^{N_t}_{j=1}\ga(X_j)$ for any $t\geq 0$ and for any real-valued $\mf{B}(0,\infty)$-measurable function $\ga$ satisfying condition  $\E_{P}\left[e^{\ga(X_{1})}\right]=1$.}\smallskip
	
For the validity of the equality $\vS=\F_\infty$ see \cite{mt2}, Remark 4.2.\smallskip

The following result of \cite{mt2}, concerning a characterization of all progressively equivalent probability measures that convert a CMRP into a CMPP, serves as a useful basic tool for our results.

\begin{prop}\label{prop1}
(See \cite{mt2}, Corollary 4.8 and Remark 4.9(c)). 
For  $P\in\hyperlink{msl}{\mathcal{M}^\ell_{S,\mathbf{K}(\vT)}}$ with $P\big(\big\{\E_P[W_1\mid\vT]<\infty\big\}\big)=1$  the following hold true:
\begin{enumerate}
	\item for any pair $(\rho,Q)\in\hyperlink{mkd}{\mathfrak M_+(D)}\times\hyperlink{msl}{\M^\ell_{S,{\bf Exp}(\Ttheta)}}$   there exists an essentially unique pair $(\be,\xi)\in\hyperlink{fpt}{\F^\ell_{P,\vT}}\times\hyperlink{r+d}{\mathcal R_+(D)}$,  such that  
	\begin{gather}
		\tag{$\mbox{rnd}(f)$}
		\ga =\ln f\text{, where $f$ is a $P_{X_1}$-a.s.  positive Radon-Nikod\'{y}m derivative of $Q_{X_1}$ with respect to $P_{X_1}$},
		\label{rndx}
	\end{gather}
	\begin{gather}
		\tag{$\mbox{rnd}(\xi)$}		
		\xi\;\;\mbox{is a Radon-Nikod\'{y}m derivative of}\;\; Q_\vT\;\;\mbox{with respect to}\;\; P_\vT,
		\label{rnd} 
	\end{gather}
	\begin{gather}
		\tag{$\ast$}
		\al(\vT)=\ln\rho(\vT)+\ln\E_{P}[W_1\mid\vT]\;\;\;P{\uph}\sigma(\vT)\mbox{-a.s.},
		\label{ast}
	\end{gather}
	and
	\begin{gather}
		\label{martPP}
		\tag{$RPM_\xi$}
		Q(A)=\int_A  M^{(\be)}_t(\vT)\,dP\quad\text{for all}\,\,\,0\leq s\leq t\,\,\text{and}\,\, A\in\F_{s},
	\end{gather}
	where 
	\[
	M^{(\be)}_t(\vT):= \xi(\vT)\cdot  \frac{e^{S_t^{(\ga)}- \Ttheta\cdot J_t}}{  1-{\bf{K}}(\vT) (J_t)}\cdot \prod_{j=1}^{N_t}\frac{d{\bf Exp}(\Ttheta)}{d{\bf K}(\vT)}(W_j), 
	\]
	with $J_t:=t-T_{N_t}$,  and the family $M^{(\be)}(\vT):=\{M^{(\be)}_t(\vT)\}_{t\in\R_+}$ is a $P$-$\text{a.s}$ positive martingale in $\mathcal{L}^1(P)$; 
	\item conversely, for any pair function $(\be,\xi)\in\hyperlink{fpt}{\F^\ell_{P,\vT}}\times\hyperlink{r+d}{\mathcal R_+(D)}$  there exists a unique pair $(\rho,Q)\in\hyperlink{mkd}{\mathfrak M_+(D)}\times\hyperlink{msl}{\M^\ell_{S,{\bf Exp}(\Ttheta)}}$  determined by  conditions \eqref{ast} and  \eqref{martPP}  and satisfying conditions  \eqref{rndx} and \eqref{rnd};
	\item  
	in both cases (i) and (ii), there exists an essentially unique rcp $\{Q_{\theta}\}_{\theta\in D}$ of $Q$ over $Q_\vT$  consistent with $\vT$ and a $P_\vT$-null set $\wt L_{\ast\ast}\in\B(D)$  satisfying for any $\theta\notin \wt L_{\ast\ast}$ conditions $Q_\theta\in\hyperlink{mstl}{{\M}^\ell_{S,{\bf Exp}(\ttheta)}}$,  \eqref{rndx}, 
	\begin{gather}
		\ttheta=\frac{e^{\al(\theta)}}{\E_{P_\theta}[W_1]},
		\tag{$\wt{\ast}$}
		\label{*}
	\end{gather}
	and
	\begin{gather}
		\tag{$RPM_\theta$}
		Q_\theta(A)=\int_{A} \wt M^{(\be)}_{t}(\theta)\,dP_\theta\quad\text{for all}\,\,\,0\leq s\leq t\,\,\text{and}\,\, A\in\F_{s},
		\label{rcp3}
	\end{gather} 
	where 
	\[
	\wt{M}^{(\be)}_t(\theta):= \frac{e^{S_t^{(\ga)}- \ttheta\cdot J_t}}{  1-{\bf{K}}(\theta) (J_t)}\cdot \prod_{j=1}^{N_t}\frac{d{\bf Exp}(\ttheta)}{d{\bf K}(\theta)}(W_j), 
	\]
	and the family $\wt{M}^{(\be)}(\theta):=\{\wt{M}^{(\be)}_t(\theta)\}_{t\in\R_+}$ is a $P_\theta$-$\text{a.s.}$ positive martingale in $\mathcal{L}^1(P_\theta)$, where $\wt L_{\ast\ast}$ is the $P_\vT$-null sets appearing in \cite{mt2}, Corollary 4.8, containing the $P_\vT$-null set $W_P\in\B(D)$ appearing in Remark \ref{rem1aa}(b).		
\end{enumerate}
\end{prop}

\section{A Characterization of Progressively Equivalent Martingale Measures for Compound Mixed Renewal Processes}\label{CRPM}

In this section  we find out a wide class of {\em canonical}   processes satisfying the condition of {\em no free lunch with vanishing risk} (written (NFLVR) for short) (see \cite{ds},  Definition 8.1.2), connecting in this way our results with this basic notion in mathematical finance.  \smallskip

In order to present the results of this section we recall the following notions. For a given real-valued process $Y:=\{Y_t\}_{t\in\R_+}$ on $(\vO,\vS)$ a probability measure $Q$ on $\vS$ is called a $\ell$-{\bf martingale measure} for $Y$, if $Y$ is a martingale in $\mathcal L^\ell(Q)$. We will say that $Y$ satisfies condition (PEMM) if there exists a 2-martingale measure $Q$ for $Y$, which is progressively equivalent to $P$.   Moreover, let $T>0$, $\mathbb T:=[0,T]$, $\F_{\mathbb T}:=\{\F_t\}_{t\in\mathbb T}$, $Q_T:=Q{\uph}\F_T$ and $Y_{\mathbb T}:=\{Y_t\}_{t\in\mathbb T}$.   We will say that the process $Y_{\mathbb T}$ satisfies condition (EMM) if there exists a 2-martingale measure $Q_T$ for $Y_{\T}$, which is equivalent to $P_T$.

\begin{symbs}\label{symb2a} 
\normalfont 
 {\bf(a)}   	For given $\be\in\hyperlink{fpt}{\F^\ell_{P,\vT}}$, denote by $\hypertarget{r+da}{\mathcal R^{\ast,\ell}_{+}(D)}:=\mathcal R^{\ast,\ell}_{+,\be}(D)$ the class of all functions $\xi\in\hyperlink{r+d}{\mathcal R_+(D)}$ such that 
\[
\xi(\vT)\cdot\Big(\frac{e^{\al(\vT)}}{\E_P[W_1\mid\vT]}\Big)^\ell\in\mathcal L^1(P),
\]
under the assumption $P\big(\big\{\E_P[W_1\mid\vT]<\infty\big\}\big)=1$. 
\smallskip 
	
\noindent	{\bf (b)} 	Denote by $\hypertarget{msla}{\M^{\ast,\ell}_{S,{\bf \Lambda}(\rho(\vT))}}$ the class of all measures $Q\in\hyperlink{msl}{\M^\ell_{S,{\bf \Lambda}(\rho(\vT))}}$ satisfying condition 
\[
\big(1/\E_Q[W_1\mid\vT]\big)^\ell\in\mathcal L^1(Q),
\] 
under the assumption $Q\big(\big\{\E_Q[W_1\mid\vT]<\infty\big\}\big)=1$. 
\smallskip

\noindent	{\bf (c}) For arbitrary $\theta\in{D}$ denote by $\hypertarget{mstla}{\M^{\ast,\ell}_{S,{\bf\Lambda}(\ttheta)}}$ the class of all probability measures $Q_\theta\in\hyperlink{mstl}{{\M}^\ell_{S,{\bf\Lambda}(\ttheta)}}$ such that $(1/\E_{Q_\bullet}[W_1])^\ell\in\mathcal{L}^1(Q_{\vT})$ under the assumption $Q_\vT\big(\big\{\E_{Q_\bullet}[W_1]<\infty\big\}\big)=1$.   
\end{symbs} 

\begin{exs}\label{rem1a} 
	\normalfont
\textbf{(a)} Inclusions  $\hyperlink{r+da}{\mathcal{R}_+^{\ast,2}(D)}\subsetneqq \hyperlink{r+da}{\mathcal{R}_+^{\ast,1}(D)}\subsetneqq \hyperlink{r+d}{\mathcal{R}_+(D)}$ hold true. \smallskip

Clearly $\hyperlink{r+da}{\mathcal{R}_+^{\ast,2}(D)}\sq \hyperlink{r+da}{\mathcal{R}_+^{\ast,1}(D)}\sq \hyperlink{r+d}{\mathcal{R}_+(D)}$.  Let $D:=(0,\infty)$ and $P\in \hyperlink{msl}{\M^\ell_{S,{\bf Exp}(\vT)}}$  and  assume  that $P_\vT$ is absolutely continuous with respect to the Lebesgue measure $\lambda$ on $\mf{B}$ restricted to $\B(D)$. Let $f_\vT$ be the corresponding   probability density functions of $\vT$ with respect to $P$. Consider the real-valued functions $\be(x,\theta):=\theta$ for all $x,\theta>0$  and $\xi(\theta):= \frac{a\cdot{e}^{-a\cdot \theta}}{f_{\vT}(\theta)}$ for each $\theta>0$,  where $a>0$ is a constant. A straightforward computation yields $\be\in\hyperlink{fpt}{\F^\ell_{P,\vT}}$ and  $\xi\in\hyperlink{r+d}{\mathcal{R}_+(D)}$. However, for $a\in(0,1]$ we have $\xi\notin\hyperlink{r+da}{\mathcal{R}^{\ast,1}_+(D)}$, while for $a\in(1,2]$ we have  $\xi\in\hyperlink{r+da}{\mathcal{R}^{\ast,1}_+(D)}\setminus\hyperlink{r+da}{\mathcal{R}^{\ast,2}_+(D)}$;  hence the required inclusions follow.  	\smallskip

\noindent \textbf{(b)} Inclusion  $\hyperlink{msla}{\M^{\ast,\ell}_{S,{\bf \Lambda}(\rho(\vT))}}\subsetneqq \hyperlink{msl}{\M^{\ell}_{S,{\bf \Lambda}(\rho(\vT))}}$ holds true. \smallskip
	
Clearly, inclusion $\hyperlink{msla}{\M^{\ast,\ell}_{S,{\bf \Lambda}(\rho(\vT))}}\sq \hyperlink{msl}{\M^{\ell}_{S,{\bf \Lambda}(\rho(\vT))}}$ holds.  Take $D$ as in (a), and assume that $\Ttheta{=}e^\vT$ and $Q\in \hyperlink{msl}{\M^\ell_{S,{\bf Exp}(\Ttheta)}}$ with  $Q_{\vT}={\bf Exp}(\eta)$, where $\eta<\ell$ is a positive constant. It then follows that $\E_Q[e^{\ell\cdot \vT}]=\infty$, implying that $Q\notin \hyperlink{msla}{\M^{\ast,\ell}_{S,{\bf Exp}(\rho(\vT))}}$. 
\end{exs}
	
\begin{rem}\label{rem1aaa}
	\normalfont
	The following statements are equivalent: 
	\begin{enumerate}
		\item  $P\in\hyperlink{msla}{\M^{\ast,\ell}_{S,{\bf K}(\vT)}}$;	
		\item  $P_\theta\in\hyperlink{mstla}{\M^{\ast,\ell}_{S,{\bf K}(\theta)}}$ with $(P_\theta)_{X_1}=P_{X_1}$  for all $\theta\notin {W_P}$, where $W_P\in\B(D)$ is the $P_\vT$-null set appearing in Remark \ref{rem1aa}(b).
	\end{enumerate}
	
	In fact, by Remark  \ref{rem1aa}(b) there exists a $P_\vT$-null set $W_P\in\mathfrak{B}(D)$ such that $P_\theta\in\hyperlink{mst}{\mathcal{M}^\ell_{S,{\bf K}(\theta)}}$ with $(P_\theta)_{X_1}=P_{X_1}$ and $\E_{P_\theta}[W_1]<\infty$  for all $\theta\notin W_P$. Taking now into account the fact 
	\[
	\int\Big(\frac{1}{\E_P[W_1\mid\vT]}\Big)^\ell\,dP=\int \Big(\frac{1}{\E_{P_\theta}[W_1]}\Big)^\ell\,P_\vT(d\theta),
	\]
	which is a consequence of \cite{lm1v}, Lemma 3.5, we get the claimed equivalence.
\end{rem}

In the next theorems we provide a characterization of all progressively equivalent martingale measures $Q$ on $\vS$ converting a CMRP under $P$ into a CMPP under $Q$, in such a way that they are associated to stochastic processes satisfying condition (NFLVR).

\begin{thm}\label{class} 
	If $P\in\hyperlink{msla}{\M^{\ast,\ell}_{S,{\bf K}(\vT)}}$ the following statements hold true:
	\begin{enumerate}
\item 
for every pair $(\rho,Q)\in\hyperlink{mkd}{\mathfrak{M}_+(D)}\times\hyperlink{msla}{\M^{\ast,\ell}_{S,{\bf Exp}(\rho(\vT))}}$ there exists an essentially unique pair $(\beta,\xi)\in\hyperlink{fpt}{\mathcal{F}^\ell_{P,\vT}}\times\hyperlink{r+da}{\mathcal{R}^{\ast,\ell}_{+}(D)}$ satisfying conditions  \eqref{rndx}, \eqref{rnd}, \eqref{ast} and \eqref{martPP}, so that $Q$ is an $\ell$-martingale measure for the process $V(\vT):=\{V_t(\vT)\}_{t\in\R_+}$ with $V_t(\vT):=S_t-t\cdot\frac{\E_P[X_1\cdot e^{\be(X_1,\vT)}\mid\vT]}{\E_P[W_1\mid\vT]}$ for any $t\geq0$;
\item
conversely, for every pair $(\beta,\xi)\in\hyperlink{fpt}{\mathcal{F}^\ell_{P,\vT}}\times\hyperlink{r+da}{\mathcal{R}^{\ast,\ell}_{+}(D)}$ there exists a unique pair $(\rho,Q)\in\hyperlink{mkd}{\mathfrak{M}_+(D)}\times\hyperlink{msl}{\M^{\ast,\ell}_{S,{\bf Exp}(\rho(\vT))}}$ determined by  conditions \eqref{ast} and \eqref{martPP}, satisfying conditions   \eqref{rndx}  and \eqref{rnd},  so that $Q$ is an $\ell$-martingale measure for $V(\vT)$;
\item
in both cases (i) and (ii), there exist a an essentially unique rcp $\Qtd$ of $Q$ over $Q_\vT$ consistent with $\vT$  satisfying for any $\theta\notin {\wt L_{\ast\ast}}$  conditions $Q_\theta\in{\M}^{\ast,\ell}_{S,{\bf Exp}(\ttheta)}$, \eqref{rndx}, \eqref{*} and \eqref{rcp3},  so that $Q_\theta$ is an $\ell$-martingale measure for the process $V(\theta):=\{V_t(\theta)\}_{t\in\R_+}$ with $V_t(\theta):=S_t-t\cdot\frac{\E_{P_{\theta}}[X_1\cdot e^{\be(X_1,\theta)}]}{\E_{P_{\theta}}[W_1]}$ for any $t\geq0$, where ${\wt L_{\ast\ast}}$ is the $P_\vT$-null set appearing in Proposition \ref{prop1}.
	\end{enumerate}	
\end{thm}

\begin{proof}  \noindent Ad (i):  Since $\hyperlink{msla}{\M^{\ast,\ell}_{S,{\bf Exp}(\rho(\vT))}}\subseteq\hyperlink{msl}{\M^{\ell}_{S,{\bf Exp}(\rho(\vT))}}$,  there exists an essentially unique pair $(\beta,\xi)\in\hyperlink{fpt}{{\F}^\ell_{P,\vT}}\times\hyperlink{r+d}{\mathcal{R}_+(D)}$ satisfying conditions \eqref{rndx}, \eqref{rnd}, \eqref{ast} and \eqref{martPP}, by Proposition \ref{prop1}(i).  Our assumption  $Q\in\hyperlink{msl}{\M^{\ast,\ell}_{S,{\bf Exp}(\rho(\vT))}}$, along with condition \eqref{ast}, yields  $\rho(\vT)\in\mathcal L^{\ell}(Q)$, implying $\xi(\vT)\cdot (e^{\al(\vT)}/\E_P[W_1\mid\vT])^\ell\in\mathcal L^1(P)$ by condition \eqref{rnd}; hence $(\beta, \xi)\in\hyperlink{fpt}{\F^{\ell}_{P,\vT}}\times\hyperlink{r+da}{\mathcal{R}^{\ast,\ell}_{+}(D)}$. Applying now \cite{lm3},  Proposition 5.1\rm{(ii)},  we get that $Q$ is an $\ell$-martingale measure for the process $V(\vT)$.\smallskip

\noindent Ad (ii): Since  $\hyperlink{fpt}{\F^{\ell}_{P,\vT}}\subseteq \hyperlink{fpt}{\F_{P,\vT}}$ and $\hyperlink{r+da}{\mathcal{R}^{\ast,\ell}_{+}(D)}\subseteq \hyperlink{r+d}{\mathcal{R}_{+}(D)}$, it follows by Proposition \ref{prop1}(ii),  that there exists a unique pair $(\rho,Q)\in\hyperlink{mkd}{\mf{M}_+(D)}\times\hyperlink{msl}{\M^{\ell}_{S,{\bf Exp}(\rho(\vT))}}$ determined by  conditions \eqref{ast} and  \eqref{martPP}   and satisfying conditions  \eqref{rndx} and \eqref{rnd}, implying, along with the assumptions of (ii), that $\xi(\vT)\cdot (e^{\al(\vT)}/\E_P[W_1\mid\vT])^\ell\in\mathcal L^1(P)$; hence $\left(1/\E_Q[W_1\mid\vT]\right)^\ell\in\mathcal L^1(Q)$. Thus, $(\rho,Q)\in\hyperlink{mkd}{\mf{M}_+(D)}\times\hyperlink{msl}{\M^{\ast,\ell}_{S,{\bf Exp}(\rho(\vT))}}$. Again by \cite{lm3},  Proposition 5.1\rm{(ii)}, we get that $Q$ is an $\ell$-martingale for $V(\vT)$.  \smallskip

\noindent Ad (iii):  In both cases (i) and (ii), by Proposition \ref{prop1}(iii), there exists an essentially unique rcp $\{Q_\theta\}_{\theta\in{D}}$ of $Q$ over $Q_\vT$ consistent with $\vT$ and a $P_\vT$-null set ${\wt L_{\ast\ast}}$, such that for any $\theta\notin {\wt L_{\ast\ast}}$ conditions $Q_\theta\in\hyperlink{mstl}{\M^\ell_{S,{\bf Exp}(\ttheta)}}$,  \eqref{rndx},  \eqref{*} and \eqref{rcp3} hold true.  But since   $Q\in\hyperlink{msla}{\M^{\ast,\ell}_{S,{\bf Exp}(\rho(\vT))}}$,  it follows by Remark \ref{rem1aaa} that $Q_\theta\in\hyperlink{mstla}{\M^{\ast,\ell}_{S,{\bf Exp}(\ttheta)}}$ for any  $ \theta\notin{\wt L_{\ast\ast}}$; hence, we can apply \cite{mt3}, Proposition 4.2,  to complete the proof of statement (iii). 
\end{proof}

\begin{thm}\label{ftap}
Let $P\in\hyperlink{msla}{\M^{\ast,2}_{S,{\bf K}(\vT)}}$. For every pair $(\be,\xi)\in\hyperlink{fpt}{{\F}^{2}_{P,\vT}}\times\hyperlink{r+da}{\mathcal{R}^{\ast,2}_{+}(D)}$ there exists a unique pair $(\rho,Q)\in\hyperlink{mkd}{\mf{M}_+(D)}\times\hyperlink{msla}{\M^{\ast,2}_{S,{\bf Exp}(\rho(\vT))}}$  determined by   conditions \eqref{ast} and \eqref{martPP} and satisfying conditions \eqref{rndx} and \eqref{rnd}, and an essentially unique rcp  $\{Q_\theta\}_{\theta\in{D}}$  of $Q$ over $Q_\vT$ consistent with $\vT$  satisfying for any $\theta\notin {\wt L_{\ast\ast}}$ conditions $Q_\theta\in\hyperlink{msla}{{\M}^{\ast,2}_{S,{\bf Exp}(\ttheta)}}$,  \eqref{rndx}, \eqref{*} and \eqref{rcp3},  so that:
\begin{enumerate}
\item the process $V_{\mathbb T}(\vT):=\{V_t(\vT)\}_{t\in\mathbb T}$ satisfies condition  (NFLVR);
		
\item  for any $\theta\notin {\wt L_{\ast\ast}}$ the process $V_{\mathbb T}(\theta):=\{V_t(\theta)\}_{t\in\mathbb T}$ satisfies condition (NFLVR),
\end{enumerate}
where ${\wt L_{\ast\ast}}\in\B(D)$ is the $P_\vT$-null set appearing in  Theorem \ref{class} and $\T:=[0,T]$ with $T>0$.
\end{thm}

\begin{proof}  Ad (i): By Theorem \ref{class}(ii) there exist a unique pair $(\rho,Q)\in\hyperlink{mkd}{\mf{M}_+(D)}\times\hyperlink{msla}{\M^{\ast,2}_{S,{\bf Exp}(\rho(\vT))}}$  determined by conditions \eqref{ast} and  \eqref{martPP}  and satisfying conditions  \eqref{rndx} and \eqref{rnd}, so that the process $V(\vT)$ is a martingale in $\mathcal L^2(Q)$; hence  for any  $T>0$ the process $V_{\mathbb T}(\vT)$ is a  $\F_{\mathbb T}$-martingale in $\mathcal L^2(Q_T)$, implying that it is a  $\F_{\mathbb T}$-semi-martingale in $\mathcal L^2(Q_T)$ (cf., e.g., \cite{vw}, Chapter 1, Section 1.3, Definition on page 23). The latter implies that $V_{\mathbb T}(\vT)$ is also a  $\F_{\mathbb T}$-semi-martingale in $\mathcal L^2(P_T)$ since $Q_T\sim P_T$ (cf., e.g., \cite{vw}, Theorem 10.1.8). Because the process $V(\vT)$ satisfies condition (PEMM), we get that the process $V_{\mathbb T}(\vT)$ must satisfy condition (EMM). Thus, we can apply the Fundamental Theorem of Asset Pricing of  Delbaen \& Schachermayer for unbounded stochastic processes, see \cite{ds}, Theorem 14.1.1, in order to conclude that the process $V_{\mathbb T}(\vT)$ satisfies condition (NFLVR). \smallskip
	
\noindent Ad (ii): By Theorem \ref{class}(iii) there exists an essentially unique rcp $\Qtd$  of $Q$ over $Q_{\vT}$ consistent with $\vT$   such that  for any $\theta\notin{\wt L_{\ast\ast}}$ conditions $Q_\theta\in\hyperlink{mstla}{\M^{\ast,2}_{S,{\bf Exp}(\ttheta)}}$,  \eqref{rndx}, \eqref{*} and \eqref{rcp3} are valid; hence  we may apply \cite{mt3}, Theorem 4.1, to obtain assertion (ii).\end{proof}

\section{An Application  to the Ruin Problem} \label{rp}
 
The fact that $\{P_\theta\}_{\theta\in D}$ is an rcp of $P$ over $P_\vT$ consistent with $\vT$, allows for the extension of some well-known results from the Poisson or renewal risk models, to their mixed counterpart. In fact,  whenever in the $P_\theta$-Cram\'{e}r-Lundberg or the $P_\theta$-Sparre Andersen risk model an explicit formula for the (infinite time) ruin probability exists, then one can just mix over the involved parameter in order to obtain explicit formulas for the corresponding mixed risk models (compare Albrecher et al. \cite{acl}, Sections 3 and 5). However, such explicit formulas, which are also computationally feasible, can  be obtained only in certain special cases, e.g., when the claim sizes follow a gamma distribution (see Constantinescu et al. \cite{csz}), a Coxian distribution (see Landriault \& Willmot \cite{lw}),  or a general phase type distribution (cf., e.g.,  Asmussen \& Albrecher \cite{asal}, Chapter IX, Theorem 4.4).  If we are interested in an exact figure for the  ruin probability in a general mixed renewal risk model,  then the only method available seems to be simulation.  In the setting of a finite horizon ruin probability, it is straightforward to use the Crude Monte Carlo method to simulate it, see \cite{asal}, page 462. The situation is more complicated for an infinite horizon ruin probability. The difficulty is that the indicator function of the ruin event in such a case cannot be simulated in finite time: no finite segment of $S$ can tell whether ruin will ultimately occur or not. In order to overcome this difficulty, some methods have been developed (cf., e.g., \cite{asal}, Chapter XV, Sections 2-5). Among them,  the most celebrated is the  change of measures technique, which give us the opportunity to express the ruin event as a quantity under the new measure  so that ruin occurs almost surely.  \smallskip

The most common change of measures techniques applied to the Sparre Andersen risk model arise from the martingales constructed via the so-called Backward or Forward Markovization Techniques for the reserve process (see Dassios \& Embrechts \cite{daem}, Section 2.3, and Dassios \cite{da}, Section 3.5,  respectively, in connection with e.g., Schmidli \cite{scm}, Sections 8.1-8.3),  where the martingales (and thus the new measures) are obtained as solutions of partial differential equations (see also \cite{mt2}, Proposition 4.15, for a simplified construction of the martingales/measures arising from the Backward Markovization Technique).  These techniques have been widely used to solve various ruin related problems (see e.g.,   Embrechts et al. \cite{esg}, Ng \& Yang \cite{ny}, Schmidli \cite{scm95, scm10} and Tzaninis \cite{t1}), as they allow for the construction of a suitable probability measure $Q^{(r)}$, where $r>0$, so that ruin occurs $Q^{(r)}$-a.s.. However, as the main assumption for their construction is the existence of the moment generating functions $M_{X_1}$ of the claim size distribution, heavy-tailed distributions (cf., e.g., \cite{rss}, Section 2.5, for the definition and their basic properties) are naturally excluded. Note that in such a case 
one can still obtain a measure $Q^{(r)}$ arising from an exponential martingale, by taking $r<0$, but in this case ruin does not occur $Q^{(r)}$-a.s. (see \cite{scm10}, Section 6, as well as our Example \ref{countercla}). The previous discussion raises the problem of constructing a probability measure $Q$ being progressively equivalent to $P$  and so the ruin occurs $Q$-a.s. but without necessarily needing the assumption that $M_{X_1}$ exists.\smallskip

In this section   we characterize all progressively equivalent   measures $Q$ on $\vS$  that convert  a  $P$-CMRP  into a $Q$-CMPP, in such a way that ruin occurs $Q$-a.s., see Theorem \ref{ruin1a}. Such a characterization allows us to find an explicit formula for the ruin probability under $P$. To this purpose we first need to prove the following   auxiliary results. \smallskip 
 
In order to justify the definition of a \textit{conditional premium density} we need the  next lemma, which  extends a  well known result in the case of a kind of  compound mixed Poisson processes (cf., e.g., \cite{gr},  Proposition 9.1).    

\begin{lem}\label{lim}
	If $P\in\hyperlink{msl}{\M^{1}_{S,{\bf K}(\vT)}}$  and  $P\big(\big\{\E_P[W_1\mid\vT]<\infty\big\}\big)=1$ then 
	\[
	\lim_{t\rightarrow\infty}\frac{S_t}{t} =\frac{\E_P[X_1]}{\E_P[W_1\mid\vT]}\quad P\text{-a.s.}.
	\]
\end{lem}
\begin{proof} 	Since $P\in\hyperlink{msl}{\M^1_{S,{\bf K}(\vT)}}$ and  $P\big(\big\{\E_P[W_1\mid\vT]<\infty\big\}\big)=1$, it follows by Remark  \ref{rem1aa}(b) that there exists a $P_\vT$-null set ${W_P}\in\mathfrak{B}(D)$ such that $P_\theta\in\hyperlink{mst}{\mathcal{M}^1_{S,{\bf K}(\theta)}}$ with $(P_\theta)_{X_1}=P_{X_1}$ and $\E_{P_\theta}[W_1]<\infty$  for any  $\theta\notin  W_P$.   Fix on an arbitrary $\theta\notin {W_P}$. We first show the validity of condition 
	\begin{equation}
		\lim_{t\rightarrow\infty}\frac{N_t}{t}=\frac{1}{\E_P[W_1\mid\vT]}\qquad P\text{-a.s.}.
		\label{41a}
	\end{equation}
	
	In fact, consider the function  $v:=\1_{\left\{\lim_{t\to\infty}\frac{N_t(\bullet)}{t}=\frac{1}{\E_{P_{\bullet}}[W_1]}\right\}}:\vO\times D\rightarrow [0,1]$ and put $g:=v\circ (id_\vO\times\vT)=\1_{\left\{\lim_{t\to\infty}\frac{N_t}{t}=\frac{1}{\E_P[W_1\mid\vT]}\right\}}$. Since $v\in\mathcal L^1(M)$, where $M:=P\circ (id_\vO\times\vT)^{-1}$, we may apply \cite{lm1v},  Proposition 3.8(i), to get  $\E_{P}\left[g\mid\vT\right]=\E_{P_\bullet}\left[v^{\bullet}\right]\circ \vT$ $P{\uph}\sigma(\vT)\text{-a.s.}$, implying     
	\begin{align*}
		P\left(\left\{\lim_{t\rightarrow\infty}\frac{N_t}{t}=\frac{1}{\E_P[W_1\mid\vT]}\right\}\right) 	&=\int\E_P[g\mid\vT]\,dP=\int\E_{P_\bullet}\left[v^{\bullet}\right]\circ \vT\,dP\\
		&=\int P_\theta\left(\left\{\lim_{t\rightarrow\infty}\frac{N_t}{t}=\frac{1}{\E_{P_\theta}[W_1]}\right\}\right) \, P_\vT(d\theta). 
	\end{align*}
	But since  $N$ is a $P_\theta$-RP$({\bf K}(\theta))$, we may apply \cite{gut},  Section 2.5, Theorem 5.1,  to get  
	\begin{equation}
		\lim_{t\rightarrow\infty}\frac{N_t}{t}=\frac{1}{\E_{P_\theta}[W_1]}\quad  P_\theta\text{-a.s..}
		\label{lim1}
	\end{equation}
The latter, along with \cite{lm1v}, Lemma 3.5(i), yields condition \eqref{41a}.\smallskip

	Since  the process $S$ is a $P_\theta$-CRP$({\bf K}(\theta),(P_\theta)_{X_1})$  with $(P_\theta)_{X_1}=P_{X_1}$,  we may apply \cite{gut},  Section 1.2, Theorem 2.3(iii), in order to get   
	\begin{equation}
		\lim_{t\rightarrow\infty}\frac{S_t}{N_t}= \E_{P_{\theta}}[X_1]\quad P_\theta\text{-a.s.};
		\label{lim0}
	\end{equation} 
	implying along with condition (a2) and \cite{lm1v},  Lemma 3.5(i), that $\lim_{t\rightarrow\infty}\frac{S_t}{N_t}=\E_P[X_1]$ $P$-a.s.. The latter, along with condition  \eqref{41a}, completes the proof.
\end{proof}

\begin{dfs}
\normalfont
Let $P\in\hyperlink{msl}{\M^{1}_{S,{\bf K}(\vT)}}$ and  $P\big(\big\{\E_P[W_1\mid\vT]<\infty\big\}\big)=1$. For any $\theta\notin {W_P}$, where $W_P\in\B(D)$ is the $P_\vT$-null set appearing in Remark  \ref{rem1aa}(b), conditions \eqref{lim1} and \eqref{lim0} imply  
\[
\lim_{t\rightarrow\infty}\frac{S_t}{t}=\frac{\E_{P_\theta}[X_1]}{\E_{P_\theta}[W_1]}\qquad P_\theta\text{-a.s.},
\]
that is, the limit $\lim_{t\rightarrow\infty}\frac{S_t}{t}$ coincides with the {\bf premium density} $p(P_\theta)$, i.e., the monetary payout per unit time, in a $P_\theta$-Sparre Andersen model  (see \cite{mt3},  page 54, for more details). Thus,   we may define a {\bf conditional premium density} for $P$ by means of 
\[
p(P,\vT):=\frac{\E_P[X_1]}{\E_P[W_1\mid\vT]} \quad P{\uph}\sigma(\vT)\text{-a.s..}
\]
In particular, for any $P\in\hyperlink{msla}{\M^{\ast,1}_{S,{\bf K}(\vT)}}$ we may define the corresponding \textbf{mixed premium density} by means of $p(P):=\E_P[p(P,\vT)]$.
\end{dfs}

\begin{rem}\label{dispd} 
\normalfont
If $P\in\hyperlink{msla}{\M^{\ast,\ell}_{S,\mathbf{K}(\vT)}}$  then the following statements are equivalent: 
\begin{enumerate}
	\item $p(P,\vT)$ is a conditional premium density  for $P$;
	\item  there exists a $P_\vT$-null set $W_{P,1}\in\B(D)$ such that $p(P,\theta)$ is the premium density for $P_\theta$, i.e., $p(P,\theta)=p(P_\theta):=\frac{\E_{P_\theta}[X_1]}{\E_{P_\theta}[W_1]}$ for any $\theta\notin W_{P,1}$. 
 \end{enumerate} 
 
In fact, first note that according to Remark \ref{rem1aaa}, there exists a $P_\vT$-null set $W_P\in\B(D)$, such that $P_\theta\in\hyperlink{mstla}{\M^{\ast,\ell}_{S,{\bf K}(\theta)}}$ with $(P_\theta)_{X_1}=P_{X_1}$   and $\E_{P_\theta}[W_1]<\infty$ for any  $\theta\notin  W_P$; hence $p(P_\theta)=\frac{\E_{P_\theta}[X_1]}{\E_{P_\theta}[W_1]}$ for any $\theta\notin W_P$. Furthermore, as statement (i) is equivalent to 
	\[
	\int_{\vT^{-1}[F]} p(P,\vT)\,dP=\int_{\vT^{-1}[F]}\frac{\E_P[X_1]}{\E_P[W_1\mid\vT]}\,dP \quad\text{for every } F\in\B(D),
	\]
	we can apply \cite{lm1v}, Lemma 3.5(i), to get
	\[
	\int_{F} p(P,\theta)\,P_\vT(d\theta)=\int_{F}\frac{\E_{P}[X_1]}{\E_{P_\theta}[W_1]}\,P_\vT(d\theta)=\int_{F}\frac{\E_{P_\theta}[X_1]}{\E_{P_\theta}[W_1]}\,P_\vT(d\theta)\quad\text{for every } F\in\B(D), 
	\]
	where the second equality follows by condition (a2),  implying that there exists a $P_\vT$-null set $W_P^{\prime}\in \B(D)$, so that $p(P,\theta)=\frac{\E_{P_\theta}[X_1]}{\E_{P_\theta}[W_1]}$ for any $\theta\notin W_P^{\prime}$. Putting now $W_{P,1}:=W_P^{\prime}\cup W_P\in\B(D)$, we get the desired equivalence of (i) and (ii). 
 \end{rem}

\begin{dfs}\label{ir}
\normalfont
Let  $S$ be an aggregate claims process induced by a counting process $N$ and a claims size process $X$. Fix on arbitrary $u>0$ and $t\geq0$ and define the function $r^u_t:\vO\times D\rightarrow\R$  by means of $r^u_t(\omega,\theta):=u+c(\theta)\cdot t-S_t(\omega)$ for any $(\omega,\theta)\in\vO\times D$, where $c$ is a positive $\mf{B}(D)$-measurable function.   For arbitrary but fixed $\theta\in D$   the process $r^u(\theta):=\{r_t^u(\theta)\}_{t\in\R_+}$, defined by means of 
$r_t^u(\theta)(\omega):=r_t^u(\omega,\theta)$ for any $\omega\in\vO$, is called the {\bf reserve process} induced by the {\bf initial reserve} $u$, the {\bf premium intensity} or {\bf premium rate} $c(\theta)$ and the aggregate claims process $S$ (cf., e.g., \cite{Sc},  pages 155--156).  The function $\psi_\theta:(0,\infty)\rightarrow[0,1]$ defined by means of $\psi_\theta(u):=P_{\theta}(\{\inf_{t\in\R_+} r^u_t(\theta)<0\})$	is called the {\bf probability of ruin} for the reserve process $r^u(\theta)$ with respect to $P_\theta$ (cf., e.g., \cite{Sc},  page 158).    \smallskip

Define the real-valued function $R_t^u(\vT)$ on $\vO$ by means of    $R_t^u(\vT):=r_t^u\circ(id_{\vO}\times\vT)$. The process $R^u(\vT):=\{R_t^u(\vT)\}_{t\in\R_+}$ is called the {\bf reserve process} induced by the initial reserve $u$, the {\bf stochastic premium intensity} or {\bf stochastic premium rate} $c(\vT)$ and the aggregate claims process $S$. The function $\psi$ defined by $\psi(u):=P(\{\inf_{ {t\in\R_+}}{R}_t^u(\vT)<0\})$ is called the {\bf probability of ruin} for the reserve process $R^u(\vT)$ with respect to $P$.  \smallskip

The  \textbf{ruin time} of the reserve process $r^u(\theta)$ is defined as $\tau_u(\theta):=\inf\{t\geq0 : r^u_t(\theta)<0\}$ (compare  e.g., \cite{scm},  page 84).  We define the \textbf{ruin time}  of the reserve process $R^u(\vT)$ by means of $T_u(\vT):=\tau_u\circ(id_\vO\times\vT)$. Recall that for the probabilities of ruin $\psi_\theta(u)$ and $\psi(u)$ for the reserve processes $r^u(\theta)$ and $R^u(\vT)$, respectively, defined above, we have 
$\psi_\theta(u)=P_\theta(\{\tau_u(\theta)<\infty\})$ and $\psi(u)=P(\{T_u(\vT)<\infty\})$,
see e.g., \cite{rss},  page 148. 
\end{dfs} 

\noindent \textit{Throughout what follows in this section, unless stated otherwise, $P\in\hyperlink{msla}{\M^{\ast,\ell}_{S,\mathbf{K}(\vT)}}$.}
 
\begin{lem}
\label{psi1}
The following statements are equivalent:
\begin{enumerate}
\item $c(\vT)\leq p(P,\vT)$ $P{\uph}\sigma(\vT)$-a.s.;
\item $\psi(u)=1$ for any $u\geq0$.
\end{enumerate}
\end{lem}

\begin{proof}
Fix on arbitrary $u\geq0$. \smallskip
	
\noindent Ad (i)$\Rightarrow$(ii):  If (i) holds, get by \cite{lm1v}, Lemma 3.5(i), the existence of a  $P_\vT$-null set $\wt{M}_P\in\mf{B}(D)$ such that $c(\theta)\leq{p}(P,\theta)$ for all $\theta\notin\wt{M}_P$.  But, due to Remark \ref{dispd}, there exists a $P_\vT$-null set $W_{P,1}\in\mf{B}(D)$ such that $p(P,\theta)=p(P_\theta)$ and $P_\theta\in\hyperlink{mstla}{\M^{\ast,\ell}_{S,{\bf K}(\theta)}}$ for all $\theta\notin{W}_{P,1}$. Putting $M_P:=\wt{M}_P\cup{W}_{P,1}\in\mf{B}(D)$  get $P_\vT(M_P)=0$ and $c(\theta)\leq{p}(P_\theta)$ for all $\theta\notin{M}_P$;  hence for any $\theta\notin M_{P}$ we may apply  \cite{Sc},  Corollary 7.1.4, which remains  true also for $u=0$, to  obtain  $\psi_\theta(u)=1$. The latter along with \cite{mt2},  Remark 3.6, yields  $\psi(u)=\int_D \psi_\theta(u)\, P_\vT(d\theta)=1$. \smallskip
	
\noindent Ad (ii)$\Rightarrow$(i): Since $N$ has zero probability of explosion we may apply \cite{Sc}, Lemma 7.1.2, which remains  true also for $u=0$, along with \cite{mt2}, Remark 3.6, to get that $\psi(u)=P\big(\{\inf_{n\in\N_0}U^u_n(\vT)<0\}\big)$, where $U^u_n(\vT):=u+\sum_{j=1}^n\big(c(\vT)\cdot W_j-X_j\big)$ for any $n\in\N_0$. But since $\psi(u)=1$,  there exists  a positive integer $n_0$ such that $U^u_{n_0}(\vT)<0$ $\,P$-a.s.; hence there exists a natural number $n_1\leq{n_0}$ such that $c(\vT)\cdot W_{n_1}-X_{n_1}<0$ $\, P$-a.s.,  implying along with condition (a2) 
\[
0\geq c(\vT)\cdot \E_P[W_{n_1}\mid\vT]-\E_P[X_{n_1}\mid\vT]=c(\vT)\cdot \E_P[W_{n_1}\mid\vT]-\E_P[X_{n_1}] \quad P{\uph}\sigma(\vT)\text{-a.s.};
\]
hence  $c(\vT)\leq \frac{\E_P[X_{n_1}]}{\E_P[W_{n_1}\mid\vT]}$ $\,P{\uph}\sigma(\vT)$-a.s.. But since $W$ is $P$-identically distributed by \cite{mt2}, Remark 2.1, and $X$ is $P$-i.i.d., it follows that statement (i) holds. 
\end{proof}

\begin{df}\label{dfnpc}
\normalfont
We say that the pair $(P,\vT)$ satisfies the {\bf conditional net profit condition} (written \hypertarget{cnpc}{$\big(\text{cnpc}(P,\vT)\big)$}  for short), if there exists a set $H\in\s(\vT)$ with $P(H)>0$ such that 
\[
c(\vT{\restr}{H})>p\big(P{\restr} \vS\cap{H},\vT{\restr}{H}\big)\quad P{\restr} \s(\vT)\cap{H}\text{-a.s.},
\]
where $\vS\cap{H}$ and $\s(\vT)\cap{H}$ are the subspace $\s$-algebras  of subsets of $H$ (see \cite{fr1}, 121A and  Notation on page 35), and $P{\restr}\vS\cap{H}$, $P{\restr}\s(\vT)\cap{H}$ are the subspace measures on $\vS\cap{H}$,  $\s(\vT)\cap{H}$, respectively (see \cite{fr1},  131B).  \smallskip

 We say that the pair $(P,\vT)$ satisfies the {\bf strong conditional net profit condition}, if 
\begin{gather}
c(\vT)>p(P,\vT)\quad P{\uph}\sigma(\vT)\text{-a.s.}  
\tag{scnpc$(P,\vT)$}
\label{scnpc}
\end{gather} 
is fulfilled. 
\end{df}

Lemma \ref{psi1} implies that in order to avoid a $P$-a.s. ruin we have to choose the stochastic premium  rate $c(\vT)$  in such a way  that the conditional net profit condition \hyperlink{cnpc}{$\big(\text{cnpc}(P,\vT)\big)$} is fulfilled. 

\begin{rem}
\normalfont
\label{dispd4}
The following statements are equivalent:
\begin{enumerate}
\item
$P$ satisfies condition  \eqref{scnpc};
\item
there exists a $P_\vT$-null set $O_P\in\B(D)$, containing the $P_\vT$-null set $W_{P,1}$ appearing Remark \ref{dispd}, such that for any $\theta\notin{O}_P$ the measure $P_\theta$ is an element of $\hyperlink{mstla}{\M^{\ast,\ell}_{S,{\bf K}(\theta)}}$    satisfying condition  $c(\theta)>p(P_\theta)$.
\end{enumerate}
	
In fact, (i) holds if and only if  there exist a $P_\vT$-null set $\wt{O}_P\in\B(D)$ such that  $c(\theta)>p(P,\theta)$ for any $\theta\notin\wt{O}_P$ by \cite{lm1v}, Lemma 3.5(i), and a $P_\vT$-null set $W_{P,1}$ in $\B(D)$  such that $P_\theta\in\M^{\ast,\ell}_{S,{\bf K}(\theta)}$ and $p(P,\theta)=p(P_\theta)$ for any $\theta\notin\wt W_{P,1}$ by Remark \ref{dispd}. Putting $O_P:=\wt{O}_P\cup {W}_{P,1}$  get (i)$\Leftrightarrow$(ii).  
\end{rem}

The next result has been proven in a more general setting   for multivariate counting processes in \cite{ja}, Proposition 3.39(a). However, as it is essential for the proof of the main result of the present section (see Theorem \ref{ruin1a}), we write it (together with its proof) exactly in the form needed for our purposes. Let $(\vO,\vS,P)$ be an arbitrary probability space. Recall that a filtration $\{\G_t\}_{t\in\R_+}$ for $(\vO,\vS)$ is called  \textbf{right-continuous} if $\G_{t+}:=\bigcap_{s>t}\G_s=\G_t$ for any $t\geq 0$.   In order to present it,  and put $\wt{\H}_0:=\{\emptyset,\vO\}$ and $\wt{\H}_n:=\s\big(\F^W_n\cup\F^X_n\big)$ for any $n\in\N$. Consider the filtration $\H:=\{\H_n\}_{n\in\N_0}$ for $(\vO,\vS)$ with $\mathcal H_n:=\s\big(\wt{\H}_n\cup\s(\vT)\big)$ for any $n\in\N_0$ and put $\mathcal H_\infty=\s\big(\bigcup_{n\in\N_0}\mathcal H_n\big)$.   In the next result it is not assumed that $N$ has zero probability of explosion. Since our interest does not exceed the information generated by the aggregate claims process $S$ and the structural parameter $\vT$, (it is not restrictive to) assume that $\vS=\F_{\infty}$.\smallskip

\begin{prop}\label{p1} 
Let $(\vO,\vS,P)$ be an arbitrary probability space, and let $S$ be an arbitrary aggregate claims process induced by a claim number process $N$ and a claim size process $X$. The  canonical filtration $\mathcal{F}$ generated by $S$ and $\vT$ is right-continuous.
\end{prop}

\begin{proof}
Fix on arbitrary $t\geq 0$. \smallskip

\noindent {\bf (a)} The equality
\[
\F_t=\Big\{\Big(\bigcup_{k\in\N_0} B_k\cap\{N_t=k\}\Big)\cup\big(B_\infty\cap\{T_\infty\leq t\}\big) : B_n\in\mathcal{H}_n \text{ for each }  n\in\overline{\N}_0\Big\} 
\]
holds true, where  $T_\infty:=\sup_{n\in\N_0}T_n$. \smallskip

In fact,  first note that since $\{T_\infty\leq t\}=\bigcap_{n\in\N_0}\{T_n\leq t\}\in\F_t$ for any $t\geq 0$, we deduce that the random variable $T_\infty$ is a  $\F$-stopping time (cf., e.g., \cite{rss}, page 404, for the definition of a stopping time).  Put 
\[
\A_t:=\Big\{\Big(\bigcup_{k\in\N_0} B_k\cap\{N_t=k\}\Big)\cup\big(B_\infty\cap\{T_\infty\leq t\}\big) : B_n\in\mathcal{H}_n \text{ for each }  n\in\overline{\N}_0\Big\} .
\]
It can be readily proven that $\mathcal{A}_t$ is a $\s$-algebra of subsets of $\vO$, and that $\{\mathcal{A}_t\}_{t\in\R_+}$ is a filtration  for $(\vO,\vS)$. To show that $\mathcal{A}_t\sq\mathcal{F}_t$, let $A\in\mathcal{A}_t$ be arbitrary. It follows that 
\[
A=\Big(\bigcup_{k\in\N_0} B_k\cap\{N_t=k\}\Big)\cup\big(B_\infty\cap\{T_\infty\leq t\}\big)	
\]
for some sets $B_n\in\mathcal H_n$ for each $n\in\ov\N_0$. But since $\mathcal H_n\subseteq \F_{T_n}$ for any $n\in\ov\N_0$, where $\F_{T_n}:=\{A\in\vS : A\cap\{T_n\leq t\}\in\F_t\text{ for all } t\geq 0\}$, we get  $B_k\cap\{N_t=k\}, B_\infty\cap\{T_\infty\leq t\}\in\F_t$, implying  $A\in\F_t$. As $A\in\A_t$ is arbitrary, we obtain $\A_t\subseteq \F_t$.
\smallskip

To show the inverse inclusion, let $E\in\bigcup_{u\in[0,t]}\big(\s(S_{u})\cup\s(\vT)\big)$ be arbitrary. If $E\in\s(\vT)$, then there exists a set $A\in\B(D)$ such that $E=\vT^{-1}[A]\in\s(\vT)$, implying $E\in\mathcal{A}_t$; hence $\s(\vT)\sq\mathcal{A}_t$. If  $E\in\bigcup_{u\in[0,t]}\s(S_{u})$  there exist a $u\in[0,t]$ such that $E\in\s(S_u)$, and so $E=S^{-1}_{u}[C]$ for some $C\in\B\big([0,\infty]\big)$, implying 
\[
E=\Big(\bigcup_{k\in\N_0} B_k\cap\{N_{u}=k\}\Big)\cup\big(B_\infty\cap\{T_\infty\leq {u}\big)\in\A_u\sq\mathcal{A}_t,
\]
where $B_k:=\big(\sum_{j=1}^k X_j \big)^{-1}[C]\in\wt{\H}_k\sq\mathcal H_k$ for all $k\in\ov{\N}_0$.  As $E$ is arbitrary we get  $\bigcup_{u\in[0,t]}\big(\s(S_{u})\cup\s(\vT)\big)\subseteq\A_t$, implying $\mathcal{F}_t\sq\mathcal{A}_t$.\smallskip

\noindent  {\bf (b)} $\H_\infty=\F_\infty$. 
\smallskip

In fact, because $\H_\infty\sq\F_{T_\infty}\sq\F_\infty$, we have $\H_\infty\sq\F_\infty$, while by (a) we get $\F_t\sq\H_\infty$ for all $t\geq 0$; hence $\F_\infty\sq\H_\infty$. \smallskip

\noindent {\bf (c)} Write $\mathcal{K}_t$ for the family of all $A\in\mathcal{H}_\infty$ such that for each $n\in\N_0$ there exists a set $A_n\in\mathcal{H}_n$ with $A\cap\{t<{T}_{n+1}\}=A_n\cap\{t<{T}_{n+1}\}$.
By standard computations it follows that $\mathcal K_t$ is a $\s$-algebra of subsets of $\vO$  and that $\mathcal K_t=\F_t$. \smallskip
   
\noindent {\bf (d)} The filtration $\mathcal F$ is right-continuous.\smallskip

Using similar arguments to those of the proof of Proposition 3.39 of \cite{ja}, we get that the filtration $\{\mathcal{K}_t\}_{t\in\R_+}$ is right-continuous, and so applying (c) we infer that $\mathcal{F}$ is right-continuous.  
\end{proof}
 
 \begin{rems}\label{r1} 
\normalfont
\textbf{(a)} If the counting process $N$ has zero probability of explosion, we get $\{T_\infty\geq{t}\}=\emptyset$; hence 
\[
\F_t=\Big\{\bigcup_{k\in\N_0} B_k\cap\{N_t=k\}: B_k\in\mathcal{H}_k \text{ for each }  k\in\N_0\Big\}. 
\] 

\noindent \textbf{(b)} The canonical filtration $\mathcal{F}^S$ of an arbitrary aggregate claims process $S$  is right-continuous. The proof runs with arguments similar to those of the proof of Proposition \ref{p1}. An alternative proof for the right-continuity of $\F^S$ works by arguments similar to those of Protter \cite{pro}, Theorem 25.\smallskip 
 

\noindent \textbf{(c)} Given an arbitrary counting process $N$, the canonical filtration  $\F^{N,\vT}$ generated by $N$ and $\vT$ is right-continuous. The proof runs with arguments similar to those of the proof of Proposition \ref{p1}.
\end{rems} 
  
\begin{lem}
\label{st}
Let $(\vO,\vS,P)$ be an arbitrary probability space. The following hold true:
\begin{enumerate}
	\item the ruin time $T_u(\vT)$ is a $\F$-stopping time;  
	\item for any $\theta\in D$ the ruin time $\tau_u(\theta)$ is a  $\F$-stopping time.
\end{enumerate}
\end{lem}
\begin{proof}
Ad (i): Let $t\geq 0$.  Since $R^u(\vT)$ has right-continuous paths  it follows that
\[
\{T_u(\vT)<t\}=\bigcup_{q\in\Q_t}\big\{R^u_q(\vT)<0\big\}\in\F_t,
\]
where $\Q_t:=\Q\cap[0,t)$ (compare \cite{rss}, Theorem 10.1.1), implying along with \cite{rss}, Lemma 10.1.1, that $T_u(\vT)$ is a  $\{\F_{t+}\}_{t\in\R_+}$-stopping time. But by Proposition \ref{p1}, $\F$ is right-continuous, and thus we may apply again \cite{rss}, Lemma 10.1.1, in order to get that $T(\vT)$ is a  $\F$-stopping time.
\smallskip

\noindent Ad (ii): Fix on arbitrary $\theta\in D$. Using the arguments of the proof of statement (i), we get that $\tau_u(\theta)$ is a   $\{\F^S_{t+}\}_{t\in\R_+}$-stopping time.  Consequently, since $\F^S$ is right-continuous by Remark \ref{r1}(b), applying \cite{rss}, Lemma 10.1.1, we get that $\tau_u(\theta)$ is a  $\F^S$-stopping time; hence it is a $\F$-stopping time since $\F^S_t\sq \F_t$ for all $t\geq 0$. 
\end{proof}

According to Theorem \ref{class}(ii), for every pair  $(\be,\xi)\in\hyperlink{fpt}{\F^\ell_{P,\vT}}\times\hyperlink{r+da}{\mathcal{R}^{\ast,\ell}_{+}(D)}$  there exists a unique pair $(\rho,Q)\in\hyperlink{mkd}{\mathfrak M_+(D)}\times\hyperlink{msla}{\mathcal{M}^{\ast,\ell}_{S,\mathbf{Exp}(\rho(\vT))}}$  determined by conditions \eqref{ast} and \eqref{martPP} and satisfying conditions \eqref{rndx}, \eqref{rnd},  so that the process $M^{(\be)}(\vT)$, involved in condition \eqref{martPP}, is a $P$-a.s. positive martingale in $\mathcal L^{1}(P)$; hence 
\[
P(A)=\E_Q\big[1/M_t^{(\be)}(\vT)\,\1_A\big]\quad\text{for all } t\geq 0\text{ and } A\in\F_t.
\] 
Since the ruin time $T_u(\vT)$ is a $\F$-stopping time by Lemma \ref{st}(i), we may apply the Optional Stopping Theorem (cf., e.g., \cite{scm}, Proposition B.2) in order to  get
\begin{gather}
\psi(u)=\E_Q\big[1/M_{T_u(\vT)}^{(\be)}(\vT)\,\1_{\{T_u(\vT)<\infty\}}\big]\quad \text{ for any }  u\geq 0.
\label{ruinq}
\end{gather}
However, the latter formula is quite complicated mainly due to the (possible) dependence between $1/M_{T_u(\vT)}^{(\be)}(\vT)$ and $\1_{\{T_u(\vT)<\infty\}}$.   Thus, we have to carefully choose an appropriate pair $(\rho,Q)\in\hyperlink{mkd}{\mathfrak M_+(D)}\times\hyperlink{msla}{\mathcal{M}^{\ast,\ell}_{S,\mathbf{Exp}(\rho(\vT))}}$   in order to eliminate such a dependence. This motivates  us to  introduce the following subclasses $\mathcal{M}^{ruin,\ell}_{S,\mathbf{Exp}(\rho(\vT))}$ and $\F^{ruin,\ell}_{P,\vT}$ of $\hyperlink{msla}{\mathcal{M}^{\ast,\ell}_{S,\mathbf{Exp}(\rho(\vT))}}$ and $\hyperlink{fpt}{{{\F}^{\ell}_{P,\vT}}}$, respectively,  which seem to be suitable for the applications of our results to the ruin problem, and to present Theorem \ref{ruin1a}.

\begin{symbs}
\label{notaruin}
\normalfont
Let  $\rho\in\hyperlink{mkd}{\mf{M}_+(D)}$. \smallskip

\noindent  \textbf{(a)} Denote by  $\hypertarget{mslr}{\mathcal{M}^{ruin,\ell}_{S,\mathbf{Exp}(\rho(\vT))}}$ the family of all probability measures  $Q\in\hyperlink{msla}{\mathcal{M}^{\ast,\ell}_{S,\mathbf{Exp}(\rho(\vT))}}$ such that 
\begin{gather}
	c(\vT)\leq{p}(Q,\vT)\;\; P{\restr}\s(\vT)\text{-a.s.,} 
	\tag{a.s.ruin$(Q,\vT)$}
	\label{ncnpc}
\end{gather} 
and by   $\hypertarget{fptr}{{\F}^{ruin,\ell}_{P,\vT}}$ the family of all functions $\beta\in\hyperlink{fpt}{\mathcal{F}^\ell_{P,\vT}}$  so that 
\begin{gather}	
	c(\vT)\leq\frac{\E_P[X_1\cdot{e}^{\beta(X_1,\vT)}\mid\vT]}{\E_P[W_1\mid\vT]}\quad P{\uph}\sigma(\vT)\mbox{-a.s.}.  
	\tag{a.s.ruin$(\be,\vT)$}	
	\label{ncnpbc}
\end{gather}

\noindent \textbf{(b)}  For given $\theta\in{D}$ denote by   $ \hypertarget{mstlr}{\mathcal{M}^{ruin,\ell}_{S,\mathbf{Exp}(\rho(\theta))}}$ the family of all probability measures $Q_\theta\in\hyperlink{mstla}{\mathcal{M}^{\ast,\ell}_{S,\mathbf{Exp}(\rho(\theta))}}$  with 
\begin{gather}	
c(\theta)\leq{p}(Q_\theta).  
\tag{a.s.ruin$(Q_\theta)$}
\label{ncnpqt}
\end{gather}

\noindent  \textbf{(c)} 
	For $\hyperlink{msla}{Q\in\M^{\ast,\ell}_{{\bf Exp}(\rho(\vT))}}$ we denote by  
	$(\hypertarget{npcpteq}{\mbox{a.s.ruin}(Q,\vT)_{\mbox{eq}}})$
	the property
	\[
	c(\vT)=p(Q,\vT)\;\; P{\restr}\sigma(\vT)\mbox{-a.s.}
	\]
	and by $\hypertarget{npcptheq}{(\mbox{a.s.ruin}(Q_\theta)_{\mbox{eq}})}$ the property
	\[
	c(\theta)=p(Q_\theta)\;\;\mbox{for}\;\; P_\vT\mbox{-a.a.}\;\; \theta\in{D}.
	\]
	By $\hypertarget{ncnpbceq}{(\mbox{a.s.ruin}(\beta,\vT)_{\mbox{eq}})}$ is denoted the property
	\[	
	c(\vT)=\frac{\E_P[X_1\cdot{e}^{\beta(X_1,\vT)}\mid\vT]}{\E_P[W_1\mid\vT]}\quad P{\uph}\s(\vT)\mbox{-a.s.}.  
	\]

\end{symbs}

\begin{ex}
\normalfont
\label{countercla}
Assume  that the pair $(P,\vT)$ satisfies condition \eqref{scnpc}. The classes $\hyperlink{fptr}{{\F}^{ruin,\ell}_{P,\vT}}$ and $\hyperlink{mslr}{\mathcal{M}^{ruin,\ell}_{S,\mathbf{Exp}(\rho(\vT))}}$ are strict subclasses of $\hyperlink{fpt}{{\F}^{\ell}_{P,\vT}}$ and $\hyperlink{msla}{\mathcal{M}^{\ast,\ell}_{S,\mathbf{Exp}(\rho(\vT))}}$, respectively.\smallskip
	
In fact, let $D:=(0,\infty)$ and $r>0$. Consider the real-valued  function $\be:=\ga+\al$ on $(0,\infty)^2$, with $\ga(x):=-r\cdot x-\ln\E_P[e^{-rX_1}]$ for each $x>0$ and $\al(\theta):=\ln\E_{P_\theta}[e^{-rX_1}]$ for each $\theta>0$. A standard computation justifies that $\E_P\big[e^{\ga(X_1)}\big]=1$ and $\E_P\big[X_1^\ell\cdot e^{\ga(X_1)}\big]<\infty$, implying that  $\be\in\hyperlink{fpt}{\F^\ell_{P,\vT}}$. However, $\be\notin\hyperlink{fptr}{\F^{ruin,\ell}_{P,\vT}}$ as
\[
\frac{\E_P[X_1\cdot{e}^{\beta(X_1,\vT)}\mid\vT]}{\E_P[W_1\mid\vT]}=\frac{\E_P[X_1\cdot{e}^{-r\cdot X_1}\mid\vT]}{\E_P[W_1\mid\vT]}<\frac{\E_P[X_1]}{\E_P[W_1\mid\vT]}<c(\vT)\quad P{\uph}\s(\vT)\text{-a.s.,}
\]
where the second inequality follows by condition \eqref{scnpc}. \smallskip

Let now  $\xi\in\hyperlink{r+da}{\mathcal{R}^{\ast,\ell}_{+}(D)}$.  Since $(\be,\xi)\in\hyperlink{fpt}{{\F}^{\ell}_{P,\vT}}\times\hyperlink{r+da}{\mathcal{R}^{\ast,\ell}_{+}(D)}$, we may apply Theorem  \ref{class}(ii) to obtain  a unique pair $(\rho,Q)\in\hyperlink{mkd}{\mf{M}_+(D)}\times\hyperlink{msl}{\M^{\ast,\ell}_{S,{\bf Exp}(\rho(\vT))}}$ determined  by  conditions \eqref{ast} and \eqref{martPP}, so that  conditions   \eqref{rndx}  and \eqref{rnd} hold. However $Q\notin\hyperlink{mslr}{\M^{ruin,\ell}_{S,{\bf Exp}(\rho(\vT))}}$ as
\[
p(Q,\vT) =\Ttheta\cdot \E_Q[X_1]=\frac{e^{\al(\vT)}\cdot \E_P[X_1\cdot e^{\ga(X_1)}\mid\vT]}{\E_P[W_1\mid\vT]} <\frac{\E_P[X_1]}{\E_P[W_1\mid\vT]}<c(\vT)\quad P{\uph}\s(\vT)\text{-a.s.,}
\]
 where the second equality follows by conditions (a2) and \eqref{ast} and  the second inequality follows again by condition \eqref{scnpc}.  The same fact remains true under the weaker condition $\big(\hyperlink{cnpc}{\text{cnpc}(P,\vT)}\big)$.
\end{ex}

\begin{rem}
	\label{dispd2} 
	\normalfont 
Let $(\rho,Q)\in\hyperlink{mkd}{\mathfrak{M}_+(D)}\times\hyperlink{msla}{\M^{\ast,\ell}_{S,{\bf Exp}(\rho(\vT))}}$  and let $\{Q_\theta\}_{\theta\in{D}}$ be the  rcp of $Q$ over $Q_\vT$ consistent with $\vT$ appearing in Theorem \ref{class}(iii). The following statements are equivalent:
	\begin{enumerate}
		\item
		$(\hyperlink{npcpteq}{\mbox{a.s.ruin}(Q,\vT)_{\mbox{eq}}})$  holds;
		\item
		there exists a $P_\vT$-null set $M_{Q,eq}\in\B(D)$, containing the $P_\vT$-null sets $\wt{L}_{\ast\ast}$ and $W_{Q,1}$ 
		appearing in Theorem \ref{class} and Remark \ref{dispd}, respectively, such that for any $\theta\notin{M}_{Q,eq}$, the  measure $Q_\theta\in\hyperlink{mstla}{\M^{\ast,\ell}_{S,{\bf Exp}(\rho(\theta))}}$ satisfies condition $(\hyperlink{npcptheq}{\mbox{a.s.ruin}(Q_\theta)_{\mbox{eq}}})$.
\end{enumerate} 

In fact, first note that by \cite{lm1v}, Lemma 3.5(i) there exists a $P_\vT$-null set $\wt{M}_Q\in\B(D)$ such that statement (i) holds if and only if   $c(\theta)=p(Q,\theta)$ for all $\theta\notin\wt{M}_Q$.  But, due to Theorem \ref{class}(iii) and Remark \ref{dispd}, there exist $P_\vT$-null sets  $\wt{L}_{\ast\ast}, W_{Q,1}\in\B(D)$, respectively, such that  $Q_\theta\in\hyperlink{mstla}{\M^{\ast,\ell}_{S,{\bf Exp}(\rho(\theta))}}$ and $p(Q,\theta)=p(Q_\theta)$  for all $\theta\notin \wt{L}_{\ast\ast} \cup{W}_{Q,1}$. Putting  $M_{Q,eq}:=\wt{M}_Q\cup \wt{L}_{\ast\ast}  \cup{W}_{Q,1}\in\B(D)$ we infer that $P_\vT(M_{Q,eq})=0$ and condition $(\hyperlink{npcptheq}{\mbox{a.s.ruin}(Q_\theta)_{\mbox{eq}}})$  holds for all $\theta\notin{M}_{Q,eq}$, i.e., (i)$\Leftrightarrow$(ii).  
\end{rem}

\begin{thm}\label{ruin1a}
If the pair $(P,\vT)$ satisfies condition \eqref{scnpc}, then the following hold true:
\begin{enumerate}
\item
for each pair $(\rho,Q)\in\hyperlink{mkd}{\mathfrak{M}_+(D)}\times\hyperlink{mslr}{\M^{ruin,\ell}_{S,{\bf Exp}(\rho(\vT))}}$ there exists an essentially unique pair $(\be,\xi)\in\hyperlink{fptr}{\F^{ruin,\ell}_{P,\vT}}\times\hyperlink{r+da}{\mathcal R^{\ast,\ell}_+(D)}$ satisfying conditions \eqref{rndx}, \eqref{rnd}, \eqref{ast} and \eqref{martPP},  so that  
 \begin{gather}
\label{asq3}
\tag{$\mbox{ruin}(P)$}	
\psi(u)= \int  \frac{1}{\xi(\vT)}\cdot   e^{-S_{T_u(\vT)}^{(\ga)}} \cdot \prod_{j=1}^{N_{T_u(\vT)}}\frac{d {\bf K}(\vT)}{d{\bf Exp}(\Ttheta)}(W_j)\,dQ\qquad (u\geq0);
\end{gather}
\item
conversely, for every pair $(\beta,\xi)\in\hyperlink{fptr}{\mathcal{F}^{ruin,\ell}_{P,\vT}}\times\hyperlink{r+da}{\mathcal{R}^{\ast,\ell}_{+}(D)}$ there exists a unique pair $(\rho,Q)\in\hyperlink{mkd}{\mathfrak{M}_+(D)}\times\hyperlink{mslr}{\M^{ruin,\ell}_{S,{\bf Exp}(\rho(\vT))}}$ determined by  conditions \eqref{ast}, \eqref{martPP} and satisfying conditions   \eqref{rndx}, \eqref{rnd}  and \eqref{asq3};
\item
in both cases (i) and (ii), there exist a an essentially unique  rcp $\Qtd$ of $Q$ over $Q_\vT$ consistent with $\vT$ and a $P_\vT$-null set $M_{\ast,Q}\in\B(D)$, containing the $P_\vT$-null set $\wt L_{\ast\ast}$ appearing in  Theorem \ref{class},  satisfying for any $\theta\notin{M}_{\ast,Q}$ conditions $Q_\theta\in\hyperlink{mstlr}{{\M}^{ruin,\ell}_{S,{\bf Exp}(\ttheta)}}$,  \eqref{rndx}, \eqref{*}, \eqref{rcp3}  
and
\begin{gather}
\label{asq2}			
\tag{$\mbox{ruin}(P_\theta)$}
\psi_\theta(u) = \int  e^{-S_{\tau_u(\theta)}^{(\ga)}} \cdot \prod_{j=1}^{N_{\tau_u(\theta)}}\frac{d {\bf K}(\theta)}{d{\bf Exp}(\ttheta)}(W_j)\,dQ_\theta\qquad  (u\geq 0).
\end{gather}
\end{enumerate}
	
In  particular, if conditions  $(\hyperlink{npcpteq}{\mbox{a.s.ruin}(Q,\vT)_{\mbox{eq}}})$ and $(\hyperlink{ncnpbceq}{\mbox{a.s.ruin}(\beta,\vT)_{\mbox{eq}}})$ hold in (i) and (ii), respectively, then $Q$ is a $\ell$-martingale measure for the reserve process $R^u(\vT)=u-V(\vT)$, and there exists a $P_\vT$-null set $\wt{M}_{\ast,Q}\in\B(D)$ containing $M_{\ast,Q}$ and the $P_\vT$-null set $M_{Q,eq}$ appearing in Remark \ref{dispd2}, such that for any $\theta\notin\wt{M}_{\ast,Q}$,  condition $(\hyperlink{npcptheq}{\mbox{a.s.ruin}(Q_\theta)_{\mbox{eq}}})$ holds, and $Q_\theta$ is a $\ell$-martingale measure for the reserve process $r^u(\theta)=u-V(\theta)$. If in addition, $\ell=2$ and $\theta\notin\wt{M}_{\ast,Q}$, then the reserve processes $R^u_{\mathbb T}(\vT):=\{R^u_t(\vT)\}_{t\in\mathbb T}$ and  $r^u_{\mathbb T}(\theta):=\{r^u_t(\theta)\}_{t\in\mathbb T}$ both satisfy the (NFLVR) property. 
\end{thm}

\begin{proof}
Fix on arbitrary $u\geq 0$. \smallskip
	
\noindent Ad (i): Since $\hyperlink{msl}{\mathcal{M}^{ruin,\ell}_{S,\mathbf{Exp}(\rho(\vT))}}\subseteq\hyperlink{msla}{\mathcal{M}^{\ast,\ell}_{S,\mathbf{Exp}(\rho(\vT))}}$, it follows by Theorem \ref{class}(i) that there exists an essentially unique pair $(\be,\xi)\in\hyperlink{fpt}{\F^\ell_{P,\vT}}\times\hyperlink{r+da}{\mathcal R^{\ast,\ell}_+(D)}$ satisfying conditions \eqref{rndx},  \eqref{rnd}, \eqref{ast} and \eqref{martPP},  so that the family $M^{(\be)}(\vT)$ is a $P$-a.s. positive martingale  in $\mathcal L^1(P)$.  The latter along with condition \eqref{ncnpc} yields $\beta\in\hyperlink{fptr}{\mathcal{F}^{ruin,\ell}_{P,\vT}}$.  Thus,  condition \eqref{ruinq},  together with condition \eqref{martPP} for $J_{T_u(\vT)}=0$,  and the fact that  ruin occurs $Q$-a.s. by  Lemma \ref{psi1},  yields condition  \eqref{asq3}. \smallskip

\noindent Ad (ii):  Since $(\be,\xi)\in\hyperlink{fpt}{\F^{\ell}_{P,\vT}}\times\hyperlink{r+da}{\mathcal R^{\ast,\ell}_+(D)}$, there exists  a unique pair $(\rho,Q)\in\hyperlink{mkd}{\mathfrak{M}_+(D)}\times\hyperlink{msla}{\M^{\ast,\ell}_{S,{\bf Exp}(\rho(\vT))}}$ determined by conditions \eqref{ast}, \eqref{martPP}  and satisfying  conditions   \eqref{rndx} and \eqref{rnd} by Theorem \ref{class}(ii). Thus, taking into account condition \eqref{ncnpbc} we obtain  \eqref{ncnpc}; hence $Q\in\hyperlink{mslr}{\mathcal{M}^{ruin,\ell}_{S,\mathbf{Exp}(\rho(\vT))}}$. Condition \eqref{asq3} follows as in (i). \smallskip
	
\noindent Ad (iii): In both cases (i) and (ii), according to Theorem \ref{class}(iii), there exist an essentially unique rcp $\Qtd$ of $Q$ over $Q_\vT$ consistent with $\vT$ and a $P_\vT$-null set $\wt L_{\ast\ast}\in\mf{B}(D)$ satisfying for each $\theta\notin\wt L_{\ast\ast}$  conditions $Q_\theta\in\hyperlink{mstla}{{\M}^{\ast,\ell}_{S,{\bf Exp}(\ttheta)}}$,  \eqref{rndx}, \eqref{*} and \eqref{rcp3}. Since condition \eqref{ncnpc}   holds, it follows as in Remark  \ref{dispd} that there exists a $P_\vT$-null set $W_{Q,1}\in\mf{B}(D)$ such that condition  \eqref{ncnpqt} is valid for each $\theta\notin W_{Q,1}$; hence $Q_\theta\in\hyperlink{mslr}{\M^{ruin,\ell}_{S,{\bf Exp}(\rho(\theta))}}$ for each $\theta\notin {M}_{\ast,Q}:=\wt L_{\ast\ast}\cup{W}_{Q,1}\in\B(D)$. Fix on arbitrary $\theta\notin{M}_{\ast,Q}$.  Condition \eqref{ncnpqt} along with \cite{Sc},  Corollary 7.1.4 yields that ruin occurs $Q_\theta$-a.s.. Taking into account condition \eqref{rcp3}, Lemma \ref{st}(ii) and the fact that ruin occurs $Q_\theta$-a.s., and using the arguments of the proof of assertion (i), we get condition \eqref{asq2}.  \smallskip
 
In  particular, if conditions $(\hyperlink{npcpteq}{\mbox{a.s.ruin}(Q,\vT)_{\mbox{eq}}})$ and $(\hyperlink{ncnpbceq}{\mbox{a.s.ruin}(\beta,\vT)_{\mbox{eq}}})$ hold in (i) and (ii), respectively, it follows by Theorem \ref{class}(i) and (ii), respectively, that $Q$ is a $\ell$-martingale measure for the process $V(\vT)$; hence for the reserve process $R^u(\vT)=u-V(\vT)$, while by Remark \ref{dispd2} there exists a $P_\vT$-null set $M_{Q,eq}\in\B(D)$ such that for any $\theta\notin{M}_{Q,eq}$  condition $(\hyperlink{npcptheq}{\mbox{a.s.ruin}(Q_\theta)_{\mbox{eq}}})$   holds, implying that we may apply Theorem \ref{class}(iii) in order to conclude that, for any $\theta\notin\wt{M}_{\ast,Q}:=M_{\ast,Q}\cup{M}_{Q,eq}\in\B(D)$,  the measure  $Q_\theta$ is a $\ell$-martingale measure for the process $V(\theta)$; hence for the reserve process $r^u(\theta)=u-V(\theta)$. If, in addition, $\ell=2$ and $\theta\notin\wt{M}_{\ast,Q}$, then by Theorem \ref{ftap}, the  reserve processes $R^u_{\mathbb T}(\vT)$ and  $r^u_{\mathbb T}(\theta)$ both satisfy the (NFLVR) property.
\end{proof}

In the next example, given $P\in\hyperlink{msla}{\M^{\ast,\ell}_{S,{\bf Exp}(\vT)}}$ such that $M_{X_1}(r):=\E_P\big[e^{r\cdot X_1}\big]<\infty$ for some $r>0$,  we obtain an explicit formula for the ruin probability under $P$ by applying Theorem \ref{ruin1a}. 

\begin{ex}
\label{ruinp} 
\normalfont
Let $D:=(0,\infty)$, $P\in\hyperlink{msla}{\M^{\ast,\ell}_{S,{\bf Exp}(\vT)}}$ such that the pair $(P,\vT)$ satisfies condition \eqref{scnpc},  and $r> 0$ so that $M_{X_1}(r)<\infty$. According to Remark \ref{dispd4}, condition \eqref{scnpc} holds if and only if there exists a $P_\vT$-null set $O_P\in\B(0,\infty)$  such that $P_\theta\in\hyperlink{mstla}{\M^{\ast,\ell}_{S,{\bf Exp}(\theta)}}$ and $c(\theta)>p(P_\theta)$ for any $\theta\notin{O}_P$. Fix on arbitrary $\theta\notin{O}_P$.\smallskip

Let $\kappa_{\theta}(r)$  be  the unique solution to the equation 
\begin{equation}\label{lem35a}
	\E_{P_\theta}\big[e^{r\cdot X_1}\big]\cdot \E_{P_\theta}\big[e^{-(\kappa_{\theta}(r)-c(\theta)\cdot{r})\cdot W_1}\big]=1,
\end{equation}
(such a solution exists by e.g. \cite{rss},  Lemma 11.5.1(a)),  define the function $\kappa:D\times\R_+\rightarrow\R$ by means of $\kappa(\theta,r):=\kappa_{\theta}(r)$ for any $(\theta,r)\in{D}\times\R_+$, 
and for fixed $r\geq 0$ denote by $\kappa_{\vT}(r)$ the random variable defined by  $\kappa_{\vT}(r)(\omega):=\kappa_{\vT(\omega)}(r)$ for any $\omega\in\vO$. Applying \cite{mt2}, Lemma 4.14, we get that  $\kappa_\vT(r)$ is the $P{\uph}\sigma(\vT)$-a.s. unique solution to the equation
\[
M_{X_1}(r)\cdot \E_P\bigl[e^{-\bigl(\kappa_{\vT}(r)+c(\vT)\cdot{r}\bigr)\cdot W_1}\mid\vT\bigr]=1\quad P{\uph}\sigma(\vT)\text{-a.s.}.
\]

Recall that  $(P_\theta)_{X_1}=P_{X_1}$ for all $\theta\notin{W}_P$ by Remark \ref{rem1aaa},  implying  $\E_{P_\theta}[e^{r\cdot{X}_1}]=\E_P[e^{r\cdot{X}_1}]$ for all $\theta\notin{O}_P\supseteq{W}_{P,1}\supseteq{W}_P$; hence condition \eqref{lem35a}  along with $P_\theta\in\hyperlink{mstla}{\M^{\ast,\ell}_{S,{\bf Exp}(\theta)}}$ yields  $\ka_\theta(r)=\theta\cdot\big(M_{X_1}(r)-1\big)-c(\theta)\cdot r$. Differentiation   with respect to $r$ yields
\begin{equation}
\ka^\prime_\theta(r)=\theta\cdot M^\prime_{X_1}(r)-c(\theta) 
\label{562}
\end{equation}
and
\[
\ka^{\prime\prime}_\theta(r)=\theta\cdot M^{\prime\prime}_{X_1}(r) 
\]
for all $r$ in a neighbourhood $[0,r_0)$ of $0$.  Since  $P(\{X_1>0\})=1$ by assumption, we get $\ka^{\prime\prime}_\theta(r)>0$ for all $r\in[0,r_0)$, implying that the function $\ka_\theta$  is strictly convex, or equivalently that $\ka_\theta^{\prime}$ is strictly increasing on on $[0,r_0)$, while condition \eqref{562} implies that   $\ka_\theta^{\prime}(0)<0$; hence there exists a unique number $r_1:=r_1(\theta)\in(0,r_0)$ so that $\ka^\prime_\theta(r_1)=0$ and $\ka^\prime_\theta(r)>0$ for all $r\in(r_1,r_0)$.
The latter condition is equivalent to condition $\ka^\prime_\vT(r)>0$ $P{\uph}\sigma(\vT)$-a.s. for any $r\in(r_1,r_0)$. Note that $\ka_\theta(r_1)=\min_{r\in(0,r_0)}\ka_\theta(r)$ and $\ka_\theta(r)>\ka_\theta(r_1)$ for any $r\in(0,r_0)$, since the restriction of $\ka_\theta$ to $[r_1,r_0)$ is strictly increasing (cf. e.g. \cite{scm},  pages 89-90).   \smallskip

Fix on $\wt{r}\in(r_1,r_0)$, consider a function $\xi\in\hyperlink{r+da}{\mathcal{R}^{\ast,\ell}_+(D)}$ and define the real-valued function $\beta$ on $(0,\infty)^2$ by means of $\beta(x,\theta):=\wt{r}\cdot{x}$ for any $(x,\theta)\in(0,\infty)^2$.  A standard computation justifies that $\be\in\hyperlink{fpt}{\F^{\ell}_{P,\vT}}$.\smallskip

Since $\ka^\prime_\vT(\wt{r})>0$ $P{\uph}\sigma(\vT)$-a.s., condition \eqref{562}, along with \cite{mt2}, Lemma 4.14,  yields  
\[
\vT\cdot \E_P\big[X_1e^{\beta(X_1,\vT)}\mid\vT\big]>c(\vT)\quad P{\uph}\sigma(\vT)\text{-a.s.},
\] 
implying that condition \eqref{ncnpbc} holds, and so $\be\in\hyperlink{fptr}{\F^{ruin, \ell}_{P,\vT}}$. Therefore, we may apply Theorem \ref{ruin1a}(ii) to get  a unique pair $(\rho,Q^{(\wt r)})\in\hyperlink{mkd}{\mathfrak{M}_+(D)}\times\hyperlink{mslr}{\M^{ruin,\ell}_{S,{\bf Exp}(\rho(\vT))}}$, determined by conditions  \eqref{ast} and \eqref{martPP} and satisfying  conditions \eqref{rndx}, \eqref{rnd} and  \eqref{asq3}; hence  for any $u\geq0$ we have
 \begin{align*}
	\psi(u) &=\int  \frac{1}{\xi(\vT)}\cdot\Big(\frac{\vT}{\Ttheta}\Big)^{N_{T_u(\vT)}}   e^{-\wt r\cdot S_{T_u(\vT)}+N_{T_u(\vT)}\cdot \ln{M}_{X_1}(\wt{r}) -\big(\vT-\Ttheta\big)\cdot\sum_{j=1}^{N_{T_u(\vT)}} W_j} \,dQ^{(\wt r)}\\
		&=\int  \frac{1}{\xi(\vT)}\cdot  e^{-\wt r\cdot S_{T_u(\vT)}  +\vT\cdot {T_u(\vT)}\cdot\big({M}_{X_1}(\wt{r})-1\big)} \,dQ^{(\wt r)}\\
		&=\E_{Q^{(\wt r)}}\bigg[\frac{e^{\wt r\cdot R^u_{T_u(\vT)}(\vT)+\ka_\vT(\wt r)\cdot {T_u(\vT)}}}{\xi(\vT)}\bigg]\cdot e^{-\wt r\cdot u},
\end{align*}
where the second one by condition \eqref{ast}. 
\end{ex}
 
One of the main drawbacks of the method used in Example \ref{ruinp} is the assumption that $M_{X_1}(r)$ exists for some $r>0$, since it excludes heavy-tailed distributions. In the following example we consider again $P\in\hyperlink{msla}{\M^{\ast,\ell}_{S,{\bf Exp}(\vT)}}$ and we demonstrate how Theorem \ref{ruin1a} can be used to cure such cases.
 
 \begin{ex}
 	\label{ruinpnmgf} 
 	\normalfont
Let $D:=(0,\infty)$, let $P\in\hyperlink{msla}{\M^{\ast,\ell}_{S,{\bf Exp}(\vT)}}$ with $P_{X_1}=\textbf{Par}(a,b)$ (cf., e.g., \cite{Sc}, page 180 for the definition of the Pareto distribution), where $a>1$ and $b>0$,  so that the pair $(P,\vT)$ satisfies condition \eqref{scnpc},  and let $z\in\mathfrak M_+(D)$ with $z(\vT)>c(\vT)$ $P$-a.s.. Consider the pair $(\be,\xi)\in \hyperlink{fpt}{{\F}^{\ell}_{P,\vT}}\times \hyperlink{r+da}{\mathcal{R}^{\ast,\ell}_+(D)}$ with  $\beta(x,\theta):=\ln z(\theta)-\ln\big(\theta\cdot \E_{P}[X_1]\big)$ for any $x,\theta>0$. Since
 	\[
 	\frac{\E_P[X_1\cdot{e}^{\beta(X_1,\vT)}\mid\vT]}{\E_P[W_1\mid\vT]}=\vT\cdot\frac{z(\vT)\cdot \E_{P}[X_1]}{\vT\cdot \E_{P}[X_1]} >c(\vT)\quad P{\uph}\s(\vT)\text{-a.s.,}
 	\]
 	we deduce that condition \eqref{ncnpbc} holds, and so $\be\in\hyperlink{fptr}{\F^{ruin, \ell}_{P,\vT}}$. Thus, we may apply Theorem \ref{ruin1a}(ii) to get  a unique pair $(\rho,Q^{(z)})\in\hyperlink{mkd}{\mathfrak{M}_+(D)}\times\hyperlink{mslr}{\M^{ruin,\ell}_{S,{\bf Exp}(\rho(\vT))}}$, determined by conditions  \eqref{ast} and \eqref{martPP} and satisfying  conditions \eqref{rndx}, \eqref{rnd} and  \eqref{asq3}; hence  for any $u\geq0$ we get
 	\begin{align*}
 \psi(u)&=\int  \frac{1}{\xi(\vT)}\cdot  \bigg(\frac{\vT\cdot\E_{P}[X_1]}{z(\vT)}\bigg)^{N_{T_u(\vT)}}\cdot  e^{-\frac{T_u(\vT)}{\E_{P}[X_1]}\cdot\big(\vT\cdot\E_{P}[X_1]-z(\vT)\big)} \,dQ^{(z)}\\
 &=\int  \frac{1}{\xi(\vT)}\cdot  \bigg(\frac{\vT\cdot b}{(a-1)\cdot z(\vT)}\bigg)^{N_{T_u(\vT)}}\cdot  e^{-\frac{(a-1)\cdot T_u(\vT)}{b}\cdot\big(\vT\cdot \frac{b}{a-1}-z(\vT)\big)} \,dQ^{(z)}.
 	\end{align*}
 \end{ex} 
 
 Note that the arguments appearing in the above example, remain true for any claim size distribution $P_{X_1}$ with finite expectations. 
 
\section{Mixed Premium Calculation Principles and Change of Measures}\label{ap}

In this section, we discuss implications of our results to the  computation of premium calculation principles in a model of an insurance market possessing the property of (NFLVR). In this context, the financial pricing of insurance (FPI for short) approach proposed by Delbaen \& Haezendonck \cite{dh} plays a key role.\smallskip

Let $T>0$. According to the  FPI approach   the liabilities of an insurance company over a fixed period of time $\T:=[0,T]$ can be represented as a price process   $U_\T:=\{U_t\}_{t\in\T}$  defined by means of $U_t:= p_t+S_t$  for any $t\in\T$, where $S_t$ represents the total amount of claims paid up to time $t$ and $p_t$ represents the total premium for the remaining risk $S_T-S_t$. Under the assumption  that the random behaviour of the price process $U_\T$ is described by the given probability measure $P$ on $\vS$, and  that the insurance market is liquid enough (see \cite{dh}, Section 1, for more details)  by applying  the Harrison-Kreps theory (see Harrison \& Kreps \cite{hk})   it follows that the existence of  a $2$-martingale measure $Q$ on $\vS$ for $U_\T$ which is equivalent to $P$ implies the elimination of arbitrage opportunities in the insurance market and vice versa (see \cite{dh}, Section 1). However, as the insurance market is not, in general, complete (see e.g. \cite{emme}, page 20,  or \cite{so}, Section 4) the measure $Q$ is not unique; hence the next step should be the selection of such a measure $Q$. \smallskip

Under the assumption 
\begin{gather}
\mbox{the aggregate process}\;\; S\;\; \mbox{is a}\;\; P\text{-CPP}(\theta_0,P_{X_1})\;\;\mbox{with}\;\; \theta_0>0, 
\tag{CPP}
\end{gather}
Delbaen and Haezendonck \cite{dh} were interested in all those measures $Q$  with linear premiums of the form
	\begin{gather}
		p_t=(T-t)\cdot p(Q)\quad\text{ for any }t\in\T,
		\tag{LP}
		\label{fpi1}
	\end{gather}
where  p(Q) is the premium density under $Q$. But as $U_\T$ is a martingale under $Q$ if and only if $U_\T-p_0=\{S_t-p(Q)\cdot{t}\}_{t\in\T}$ is so, Delbaen \& Haezendonck  faced the problem of characterizing all those risk-neutral measures $Q$ on $\vS$ under which the compensator of $S_\T$ is a linear deterministic function of $t\in\T$.	As the linearity of the premiums implies that $S_\T$ is a CPP under $Q$ and vice versa (see \cite{dh},  Section 1), the above problem is equivalent to the following one: \smallskip

{\em If $S$ is a CPP under $P$, then characterize all progressively equivalent to $P$ probability measures $Q$ on $\vS$ such that $S$ remains a CPP under $Q$.}\smallskip

Recall that under the FPI framework,   a {\bf premium calculation principle} (PCP for short) is a  probability measure $Q$ on  $\vS$ which is progressively equivalent to $P$, the process $S$ is a $Q$-CPP and $X_1\in\mathcal{L}^1(Q)$, see \cite{dh}, Definition 3.1.   If   the distribution $Q_\vT$   is not degenerate,  then the probability measure $Q\in\hyperlink{msla}{\M^{\ast,1}_{S,{\bf Exp}(\Ttheta)}}$ constructed in  Theorem \ref{class}(ii) fails to be a PCP. Nevertheless, by virtue of  Theorem \ref{class}(iii) there exists an essentially unique rcp $\{Q_\theta\}_{\theta\in D}$ of $Q$ over $Q_\vT$  consistent with $\vT$  such that for $P_\vT$-a.a. $\theta\in{D}$ the probability measures $Q_\theta$ are PCPs. Thus, it seems natural to call every probability measure $Q\in\hyperlink{msla}{\M^{\ast,1}_{S,{\bf Exp}(\rho(\vT))}}$ a {\bf mixed PCP}.   \smallskip

In order for a mixed PCP to provide a realistic and viable pricing framework it should give more weight to unfavourable events in a risk-averse environment, i.e., conditions 
	\begin{equation}
		p(P,\vT)<p(Q,\vT)<\infty\quad P{\uph}\sigma(\vT)\text{-a.s.}.
		\label{pmpcp1}
	\end{equation}
	and 
	\begin{equation}
		p(P)<p(Q)<\infty 
		\label{pmpcp2}
	\end{equation}
	must hold true.  
	
	\begin{rem}
		\label{dispd2a} 
		\normalfont 
Let $(\rho,Q)\in\hyperlink{mkd}{\mathfrak{M}_+(D)}\times\hyperlink{msla}{\M^{\ast,\ell}_{S,{\bf Exp}(\rho(\vT))}}$  and let $\{Q_\theta\}_{\theta\in{D}}$ be the  rcp of $Q$ over $Q_\vT$ consistent with $\vT$ appearing in Theorem  \ref{class}(iii). The following statements are equivalent:
		\begin{enumerate}
			\item
			$p(P,\vT)<p(Q,\vT)<\infty\quad P{\uph}\s(\vT)$-a.s.;
			\item
			there exists a $P_\vT$-null set $\wt M_{P,Q}\in\B(D)$ containing the $P_\vT$-null sets $\wt{L}_{\ast\ast}$ and $W_{P,1}\cup W_{Q,1}$ appearing in Theorem \ref{class}(iii) and Remark \ref{dispd}, respectively,  such that $Q_\theta\in\hyperlink{mstla}{\M^{\ast,\ell}_{S,{\bf Exp}(\rho(\theta))}}$ and $p(P_\theta)< p(Q_\theta)<\infty$ for any $\theta\notin \wt M_{P,Q}$.
		\end{enumerate}
		
In fact, first note that by \cite{lm1v}, Lemma 3.5(i), there exists a $P_\vT$-null set $M_{P,Q}\in\B(D)$ such that statement \rm{(i)} holds  if and only if  $p(P,\theta)<p(Q,\theta)<\infty$ for all $\theta\notin{M}_{P,Q}$. Putting $\wt M_{P,Q}:=\wt{L}_{\ast\ast}\cup{M}_{P,Q}\cup W_{P,1}\cup W_{Q,1}\in\mf{B}(D)$ and applying Theorem \ref{class}(iii) and Remark \ref{dispd} we infer that $p(P_\theta)<p(Q_\theta)<\infty$ for all $\theta\notin \wt M_{P,Q}$, i.e., (i)$\Leftrightarrow$(ii). 
	\end{rem}

However, the existence of a mixed PCP  does not, in general, guarantee the validity of conditions \eqref{pmpcp1} and \eqref{pmpcp2} as the next two examples demonstrate. In the first example, we construct a mixed PCP that does not satisfy condition  \eqref{pmpcp1}  and leads to a $P$-a.s. ruin.

\begin{ex}
	\label{counter1}
	\normalfont
Let $P\in\hyperlink{msla}{\M^{\ast,\ell}_{S,{\bf K}(\vT)}}$, $D:=(0,\infty)$ and $r>0$. Consider the pair $(\beta,\xi)\in\hyperlink{fpt}{\mathcal{F}^\ell_{P,\vT}}\times\hyperlink{r+da}{\mathcal{R}^{\ast,\ell}_{+}(D)}$ with  $\be(x,\theta):=-r\cdot x$ for all $x, \theta>0$.  Applying  Theorem \ref{class}(ii) we obtain  a unique pair $(\rho,Q)\in\hyperlink{mkd}{\mf{M}_+(D)}\times\hyperlink{msla}{\M^{\ast,\ell}_{S,{\bf Exp}(\rho(\vT))}}$ determined by conditions \eqref{ast} and \eqref{martPP}, satisfying conditions \eqref{rndx}, \eqref{rnd}, and so that $Q$ is an $\ell$-martingale measure for the  reserve process $R^u(\vT)$ with $c(\vT)=p(Q,\vT)=\frac{\E_P[X_1\cdot e^{-r\cdot X_1}\mid\vT]}{\E_P[W_1\mid\vT]}$ $\,\,P{\uph}\sigma(\vT)$-a.s.. However, for this particular choice of mixed PCP one has that $p(Q,\vT)<p(P,\vT)$ $\,P{\uph}\sigma(\vT)$-a.s. which according to Lemma \ref{psi1} leads to an a.s. ruin. 
\end{ex}

In the next example, we construct a mixed PCP for which condition \eqref{pmpcp1} holds but condition \eqref{pmpcp2} fails.

\begin{ex}
	\label{counter2}
	\normalfont
Let $D:=(0,\infty)$ and $P\in\hyperlink{msla}{\M^{\ast,\ell}_{S,{\bf Exp}(\vT)}}$ with $P_\vT={\bf Exp}(1)$. Consider the pair $(\beta,\xi)\in\hyperlink{fpt}{\mathcal{F}^\ell_{P,\vT}}\times\hyperlink{r+da}{\mathcal{R}^{\ast,\ell}_{+}(D)}$ with  $\be(x,\theta):=\ln2$ for all $x,\theta>0$ and $\xi(\theta):=4\cdot e^{-3\cdot \theta}$ for any $\theta>0$.  Applying Theorem \ref{class}(ii) we obtain  a unique pair $(\rho,Q)\in\hyperlink{mkd}{\mf{M}_+(D)}\times\hyperlink{msla}{\M^{\ast,\ell}_{S,{\bf Exp}(\rho(\vT))}}$ determined by conditions \eqref{ast} and \eqref{martPP}, satisfying conditions \eqref{rndx}, \eqref{rnd},  and so that $Q$ is an $\ell$-martingale measure for the reserve process $R^u(\vT)$  with $c(\vT)=p(Q,\vT)=2\cdot\vT\cdot\E_P[X_1]$ $\,P{\uph}\sigma(\vT)$-a.s..  Even though the conditional premium densities satisfy condition \eqref{pmpcp1}, the corresponding premium densities satisfy the converse inequality since $p(Q)=\frac{\E_P[X_1]}{2}$ and  $p(P)=\E_P[X_1]$. 
\end{ex}

 Examples \ref{counter1} and \ref{counter2} raise the question when a mixed PCP satisfies conditions \eqref{pmpcp1} and \eqref{pmpcp2}  or the implication \eqref{pmpcp1}$\Rightarrow$\eqref{pmpcp2}.  
 
 \begin{rem}\label{rem52} \normalfont 
 If $P\in\hyperlink{msla}{\M^{\ast,\ell}_{S,{\bf K}(\vT)}}$ such that the pair $(P,\vT)$ satisfies condition \eqref{scnpc}, then for every pair $(\be,\xi)\in \hyperlink{fptr}{{\F}^{ruin,\ell}_{P,\vT}}\times\hyperlink{r+da}{\mathcal R^{\ast,\ell}_+(D)}$, with $\be$ satisfying condition $(\hyperlink{npcpteq}{\mbox{a.s.ruin}(Q,\vT)_{\mbox{eq}}})$, there exists a unique pair  $(\rho,Q)\in\hyperlink{mkd}{\mf{M}_+(D)}\times\hyperlink{msla}{\M^{\ast,\ell}_{S,{\bf Exp}(\rho(\vT))}}$ so that condition \eqref{pmpcp1} holds true. In fact, this follows immediately by Theorem \ref{ruin1a}.
 \end{rem}

 In the next proposition we find  sufficient conditions for the validity of the implication  \eqref{pmpcp1}$\Rightarrow$\eqref{pmpcp2}.  To prove it, we need the following lemma, which is a consequence of Schmidt \cite{Sc1}, Theorem 2.2, but we write it exactly in the form needed for our purposes.

\begin{lem}  
\label{mon}
Let $(\vO,\vS,P)$ be an arbitrary probability space. If $Z:\vO\rightarrow\R$ is a random variable, $J\in\mf{B}$ a Borel set satisfying $P(\{Z\in{J}\})=1$, and $f,g:J\rightarrow\R$  monotonic functions of the same monotonicity which are either positive or for which $f(Z),g(Z),f(Z)\cdot{g}(Z)\in\mathcal L^{1}(P)$, then   	
\[
\E_P[f(Z)\cdot g(Z)]\geq\E_P[f(Z)]\cdot\E_P[g(Z)].
\]	
If the functions $f$, $g$ have different 
monotonicity, then the inequality continues to apply but in the opposite direction.	
\end{lem}

\begin{prop}
\label{mpcp}
Let $D:=(0,\infty)$, $P\in\M^{\ast,\ell}_{S,{\bf K}(\vT)}$, $(\rho,Q)\in\hyperlink{mkd}{\mathfrak{M}_+(D)}\times\hyperlink{msla}{\M^{\ast,\ell}_{S, {\bf \Lambda}(\rho(\vT))}}$ 
and let $\xi\in\hyperlink{r+da}{\mathcal{R}^{\ast,\ell}_{+}(D)}$, $\wt L_{\ast\ast}$, $\{Q_\theta\}_{\theta\in D}$ be as in Theorem \ref{class}. If 
for all $\theta\notin\wt L_{\ast\ast}$ the functions $\theta\mapsto\xi(\theta)$ and  $\theta\mapsto p(Q_\theta)$ 
are  monotonic of the same monotonicity, then
\begin{enumerate}
\item
$p(P_\theta)\leq p(Q_\theta)$ for $P_\vT$-a.a. $\theta>0$, implies $p(P)\leq p(Q)<\infty$. 
\item
$p(P_\theta)<p(Q_\theta)$ for $P_\vT$-a.a. $\theta>0$ implies $p(P)< p(Q)<\infty$.	
\end{enumerate}	
	
\end{prop}
\begin{proof}  
First note that given $P\in\hyperlink{msla}{\M^{\ast,\ell}_{S,{\bf K}(\vT)}}$ and  $(\rho,Q)\in\hyperlink{mkd}{\mathfrak{M}_+(D)}\times\hyperlink{msla}{\M^{\ast,\ell}_{S, {\bf \Lambda}(\rho(\vT))}}$ the existence of a function $\xi\in\hyperlink{r+da}{\mathcal{R}^{\ast,\ell}_{+}(D)}$ follows by Theorem \ref{class}, according to which there exists an essentially unique pair $(\beta,\xi)\in\hyperlink{fpt}{\mathcal{F}^\ell_{P,\vT}}\times\hyperlink{r+da}{\mathcal{R}^{\ast,\ell}_{+}(D)}$ satisfying among others condition \eqref{ast}.\smallskip

\noindent Ad (i):  Since  for all $\theta\notin \wt L_{\ast\ast}$   the functions $\theta\mapsto\xi(\theta)$ and  $\theta\mapsto p(Q_\theta)$  are monotonic of the same monotonicity, if  $p(P_\theta)\leq p(Q_\theta)$  for $P_\vT$-a.a. $\theta>0$  we get 
\begin{align*}
p(P)&=\E_{P_\vT}[p(P_\theta)]\leq\E_{P_\vT}[p(Q_\theta)]
		=\E_{Q_\vT}[p(Q_\theta)\cdot(\xi(\theta))^{-1}]\\
		&\leq \E_{Q_\vT}    
		[p(Q_\theta)]\cdot\E_{Q_\vT}[(\xi(\theta))^{-1}]
		=\E_{Q_\vT}[p(Q_\theta)]<\infty,	
\end{align*}	
where the second inequality follows by Lemma \ref{mon}, the last equality is a consequence of the fact that $\xi$ is a Radon-Nikodym derivative of $Q_\vT$ with respect to $P_\vT$, and the last inequality follows by   $p(Q_\bullet)\in\mathcal{L}^1(Q_\vT)$; hence $p(P)\leq p(Q)<\infty$.\smallskip

\noindent Ad (ii):
 If $p(P_\theta)<p(Q_\theta)$ for $P_\vT$-a.a. $\theta>0$, then  $\E_{P_\vT}[p(P_\theta)]<\E_{P_\vT}[p(Q_\theta)]$. \smallskip

The latter along with the arguments of the proof of (i) yields $p(P)=\E_{P_\vT}[p(P_\theta)]<\E_{P_\vT}[p(Q_\theta)]={p}(Q)<\infty$. This completes the proof.
\end{proof}

\section{Examples}\label{ex}

In this section, applying our results, we provide some examples to show how to construct mixed PCPs $Q$ satisfying conditions  \eqref{pmpcp1} and \eqref{pmpcp2} and such that for any $T>0$ the processes  $V_{\mathbb T}(\vT)$  and $R_{\mathbb T}^u(\vT)$, for $u>0$, have the property
of (NFLVR). Moreover,  we provide explicit formulas for the ruin probability for the reserve process $R^u(\vT)$ with respect to the original measure $P$.

\begin{ex}\label{expcp1}
\normalfont
Take $D:=(1,\infty)^2$, and let $\vT=(\vT_1,\vT_2)$ be a $D$-valued random vector on $\vO$ with $\vT_1, \vT_2\in\mathcal L^1(P)$.  Moreover, assume that $P\in\hyperlink{msla}{\M^{\ast,2}_{S,{\bf K}(\vT)}}$ with 
	\[
	{\bf K}(\vT):=\frac{1}{2}\cdot{\bf Exp}(1/\vT_1)+\frac{1}{2}\cdot{\bf Exp}(1/\vT_2),
	\]
	and $P_{X_1}={\bf Ga}(\zeta,2)$ for $\zeta>0$.  \smallskip
	
	Consider the real-valued function $\be:=\ga+\al$ on $(0,\infty)\times D$,  with $\ga(x):=\ln\frac{\E_P[X_1]}{2c}-\ln x+\frac{2(c-1)}{c\E_P[X_1]}\cdot x$ for any $x>0$ with $c>2$ a  constant, and $\al(\theta):=0$ for any $\theta\in D$. Applying standard computations, we obtain that  $\E_P[e^{\ga(X_1)}]=1$ and $\E_P[X^2_1\cdot e^{\ga(X_1)}]=\frac{2c^2}{\zeta}<\infty$, implying  $\be\in\hyperlink{fpt}{\F^{2}_{P,\vT}}$. \smallskip

Let  $\xi\in\hyperlink{mkd}{\mathfrak M_+(D)}$ defined by means of $\xi(\theta):=\xi(\theta_1,\theta_2):=1$ for any $\theta\in D$. Clearly $\E_P[\xi(\vT)]=1$, implying $\xi\in\hyperlink{r+d}{\mathcal R_+(D)}$, while $P\in\mathcal{M}^{\ast,\ell}_{S,{\bf K}(\vT)}$ yields 
\[
\E_P\left[\left(\frac{2}{\vT_1+\vT_2}\right)^2\cdot\xi(\vT)\right]<\infty,
\] 
and so $\xi\in\hyperlink{r+da}{\mathcal R^{\ast,2}(D)}$.\smallskip
	
\noindent {\bf (a)}	Since $(\be,\xi)\in\hyperlink{fpt}{\F^2_{P,\vT}}\times\hyperlink{r+da}{\mathcal R^{\ast,2}(D)}$,  by  Theorem  \ref{class}   there exist  a unique pair $(\rho,Q)\in\hyperlink{mkd}{\mathfrak M_+(D)}\times \hyperlink{msla}{\M^{*,2}_{S,{\bf Exp}(\Ttheta)}}$ determined by conditions \eqref{ast} and \eqref{martPP}, such that conditions  \eqref{rndx}  and \eqref{rnd} are valid, an essentially unique rcp $\{Q_\theta\}_{\theta\in D}$ of $Q$ over $Q_\vT$ consistent with $\vT$ and a $P_\vT$-null set $\wt L_{\ast\ast}\in\B(D)$ such that for any  $\theta\notin \wt L_{\ast\ast}$ conditions $Q_\theta\in\hyperlink{mstla}{{\M}^{\ast,2}_{S,{\bf Exp}(\ttheta)}}$, \eqref{rndx}, \eqref{*} and \eqref{rcp3} hold true.  It then follows by conditions \eqref{ast} and \eqref{rndx} that $\rho(\vT)=\frac{2}{\vT_1+\vT_2}$ $P{\uph}\sigma(\vT)$-a.s. and $Q_{X_1}={\bf Exp}(\zeta/c)$, respectively,  while  condition \eqref{rnd} yields
\begin{equation}
	Q_\vT(B)=\E_P[\1_{\vT^{-1}[B]}\cdot\xi(\vT)]=P_\vT(B)\quad\text{for any } B\in\B(D).
\label{611}
\end{equation}
Thus, for any $\theta\notin \wt L_{\ast\ast}$ the probability measure $Q_\theta$ is a PCP satisfying condition
\begin{equation}\label{eq14}
p(P_\theta)=\frac{4}{\zeta\cdot(\theta_1+\theta_2)}<\frac{2\cdot c}{\zeta\cdot(\theta_1+\theta_2)}=p(Q_\theta)<\infty;
\end{equation}
hence condition \eqref{pmpcp1} holds by Remark \ref{dispd2a}. Conditions  \eqref{611} and \eqref{eq14} imply   condition \eqref{pmpcp2}. \smallskip
	
\noindent  {\bf (b)} Again by Theorem \ref{class}, the   measure $Q$ is a $2$-martingale measure for the process $V(\vT)$ with 
\begin{equation}
V_t(\vT)=S_t-t\cdot\frac{2\cdot c}{\zeta\cdot(\vT_1+\vT_2)}\quad\text{for any } t\geq 0,
\label{11}	
\end{equation}
and for any $\theta\notin\wt L_{\ast\ast}$ the probability measure $Q_\theta$ is a $2$-martingale measure for the process $V(\theta)$ with $V_t(\theta)=S_t-t\cdot\frac{2\cdot c}{\zeta\cdot(\theta_1+\theta_2)}$ for any $t\geq 0$.  In particular, for any $T>0$, Theorem \ref{ftap} 	asserts that both processes $V_{\mathbb T}(\vT)$ and $V_{\mathbb T}(\theta)$ satisfy condition (NFLVR). \smallskip

\noindent {\bf (c)} Consider the reserve process $R^u(\vT):=u-V(\vT)$ ($u>0$). The equality $c(\vT)=p(Q,\vT)$, together with (a), implies that condition  \eqref{scnpc} is valid and that $Q\in\hyperlink{mslr}{\M^{ruin,2}_{S,{\bf Exp}(\rho(\vT))}}$; hence by Theorem \ref{ruin1a} we get that $Q$ is a $2$-martingale measure for the reserve process $R^u(\vT)$,  ruin occurs $Q$-a.s. and   
\[
\psi(u)=\E_Q\big[C_1\big(N_{T_u(\vT)}, W,X,\vT\big)\big]\cdot e^{-\frac{\zeta(c-1)}{c}\cdot u},
\]
where 
\begin{align*}
C_1\big(N_{T_u(\vT)}, W,X,\vT\big) :=\Big(\prod_{j=1}^{N_{T_u(\vT)}} \frac{c\cdot\zeta}{\Ttheta}\cdot  X_j\Big)\cdot e^{ \frac{\zeta(c-1)}{c}\cdot R^u_{T_u(\vT)}(\vT) -\rho(\vT)\cdot(c-2)\cdot {T_u(\vT)} +\sum_{j=1}^{N_{T_u(\vT)}}\ln\Big(\frac{1}{2\vT_1}\cdot e^{-\frac{W_j}{\vT_1}}+\frac{1}{2\vT_2}\cdot e^{-\frac{W_j}{\vT_2}}\Big)}.
\end{align*}
The latter follows by the equalities
\[
S_{T_u(\vT)}^{(\ga)}=N_{T_u(\vT)}\cdot\ln\Big(\frac{\E_P[X_1]}{2c}\Big)-\sum_{j=1}^{N_{T_u(\vT)}}\ln X_j+\frac{2(c-1)}{c\E_P[X_1]}\cdot S_{T_u(\vT)}
\]
and
\[
\prod_{j=1}^{N_{T_u(\vT)}}\frac{d {\bf K}(\vT)}{d{\bf Exp}(\Ttheta)}(W_j)=e^{\sum_{j=1}^{N_{T_u(\vT)}}\ln\Big(\frac{1}{2\vT_1}\cdot e^{-\frac{W_j}{\vT_1}}+\frac{1}{2\vT_2}\cdot e^{-\frac{W_j}{\vT_2}}\Big) -N_{T_u(\vT)}\cdot \ln\Ttheta+\Ttheta\cdot{T(\vT)}}
\] 
along with conditions \eqref{asq3}, \eqref{11} and \eqref{ast}.  \smallskip
 
Moreover, again by Theorem \ref{ruin1a}, for any $\theta\notin
\wt{M}_{\ast,Q}$, the reserve processes $R^u_{\mathbb T}(\vT):=\{R^u_t(\vT)\}_{t\in\mathbb T}$ and  $r^u_{\mathbb T}(\theta):=\{r^u_t(\theta)\}_{t\in\mathbb T}$ both satisfy the (NFLVR) property.
\end{ex}

In our next example we rediscover  Shaun Wang's risk-adjusted premium principle (see \cite{wangs}, for the definition and its properties). 

\begin{ex}\label{expcp3}
	\normalfont
Let $D:=(0,\infty)$  and assume that $P\in\hyperlink{msla}{\M^{\ast,2}_{S,{\bf Exp}(\vT)}}$ such that $P_{X_1}$ is absolutely continuous with respect to the Lebesgue measure $\lambda$ on $\mf{B}$ restricted to $\B(0,\infty)$. Denote by   $\wb{F}_{X_1}(x):=P_{X_1}\left((x,\infty)\right)$ for any $x>0$ the corresponding  survival function of the random variable $X_1$, and  assume that $\int_0^\infty x\cdot \left(\wb{F}_{X_1}(x)\right)^{\frac{1}{c}}\,\lambda(dx)<\infty$, where $c>1$ is a constant. Recall that the {\bf risk adjusted premium} for  $X_1$ is defined by means of
	\[
	\pi_c(X_1):=\int_0^\infty \left(\wb{F}_{X_1}(x)\right)^{\frac{1}{c}}\,\lambda(dx)\quad\text{ for any }\,\,c\geq 1.
	\]   
(see \cite{wangs}, Definition 2).\smallskip
	
Consider the real-valued function $\be:=\ga+\al$ on $(0,\infty)^2$, with $\ga(x):=-\ln c+(\frac{1}{c}-1)\cdot \ln \wb{F}_{X_1}(x)$ for any $x>0$ and  $\al(\theta):=0$ for any $\theta>0$.  By standard computations,  get   $\E_P[e^{\ga(X_1)}]=1$ and $\E_P[X^2_1\cdot e^{\ga(X_1)}]=2\cdot\int_0^\infty x\cdot \left(\wb{F}_{X_1}(x)\right)^{\frac{1}{c}} \,\lambda(dx)<\infty$, implying  $\pi_c(X_1)<\infty$ and  $\be\in\hyperlink{fpt}{\F^{2}_{P,\vT}}$.\smallskip
	
For any $r\in{E}_\vT:=\{\wt{r}\geq{0}: M_\vT(\wt{r}):=\E_P[e^{\wt{r}\cdot\vT}]<\infty\}$, where by $M_\vT$ is denoted the moment generating function of $\vT$, define the function $\xi\in\hyperlink{mkd}{\mathfrak M_+(D)}$  by means of  $\xi(\theta):=\frac{e^{r\cdot\theta}}{M_\vT(r)}$ for any $\theta>0$. Clearly $\E_P[\xi(\vT)]=1$, implying that $\xi\in\hyperlink{r+d}{\mathcal R_+(D)}$. But since  $P\in\hyperlink{msla}{\M^{\ast,2}_{S,{\bf Exp}(\vT)}}$,  get  $\E_P[\vT^2]<\infty$, implying $M_\vT^{\prime\prime}(r)<\infty$  for all $r$ in a neighbourhood of $0$ in $E_\vT$; hence 
\[
\E_P[\vT^2\cdot\xi(\vT)]=\frac{\E_{P}[\vT^2\cdot e^{r\cdot\vT}]}{M_\vT(r)}=\frac{M_\vT^{\prime\prime}(r)}{M_\vT(r)}<\infty,
\]
implying  $\xi\in\hyperlink{r+da}{\mathcal R^{\ast,2}_+(D)}$.\smallskip
	
\noindent	{\bf (a)}	Since $(\be,\xi)\in \hyperlink{fpt}{\F^{2}_{P,\vT}}\times\hyperlink{r+da}{\mathcal R^{\ast,2}_+(D)}$, we may apply Theorem \ref{class} in order to get  a unique pair $(\rho,Q)\in\hyperlink{mkd}{\mathfrak{M}_+(D)}\times\hyperlink{msla}{\M^{\ast,2}_{S,{\bf Exp}(\rho(\vT))}}$ determined by  conditions \eqref{ast} and  \eqref{martPP}, so that conditions  \eqref{rndx}  and \eqref{rnd} hold, an essentially unique rcp $\{Q_\theta\}_{\theta>0}$ of $Q$ over $Q_\vT$ consistent with $\vT$ and a $P_\vT$-null set $\wt L_{\ast\ast}\in\B(0,\infty)$ satisfying for any $\theta\notin \wt L_{\ast\ast}$ conditions $Q_\theta\in\hyperlink{mstla}{{\M}^{\ast,2}_{S,{\bf Exp}(\ttheta)}}$, \eqref{rndx}, \eqref{*} and \eqref{rcp3}.  It then follows by conditions \eqref{ast} and \eqref{rndx} that  $\rho(\vT)=\vT$ $P{\uph}\s(\vT)$-$\text{a.s.}$ and     
	\[ 
	Q_{X_1}(A)=\E_P[\1_{X^{-1}_1[A]}\cdot e^{\ga(X_1)}]=\int_{A} \frac{1}{c}\cdot  \big(\wb{F}(x)\big)^{\frac{1}{c}-1}\,P_{X_1}(dx)\quad\text{for any } A\in\B(0,\infty),
	\]
respectively, while by condition \eqref{rnd}  for any $B\in\B(0,\infty)$  we get	
	\begin{equation}\label{eq15a}	
	Q_\vT(B)=\E_P[\1_{\vT^{-1}[B]}\cdot\xi(\vT)]=\int_{B} \frac{e^{r\cdot\theta}}{M_\vT(r)}\,P_\vT(d\theta)\quad\text{for } r\in{E}_\vT;
	\end{equation}
hence for any $\theta\notin \wt L_{\ast\ast}$ the corresponding measure $Q_\theta$ is a PCP satisfying condition 
\[
p(P_\theta)=\theta\cdot\int_0^\infty \wb{F}(x) \,\lambda(dx)<\theta\cdot\int_0^\infty \left(\wb{F}(x)\right)^{\frac{1}{c}}\,\lambda(dx)=\theta\cdot\pi_c(X_1)=p(Q_\theta)<\infty,
\]
implying  condition \eqref{pmpcp1} by  Remark \ref{dispd2a}. Since for all $\theta\notin\wt L_{\ast\ast}$ the functions $\theta\mapsto\xi(\theta)$ and $\theta\mapsto p(Q_\theta)$ are monotonic   of the same monotonicity, we may apply Proposition \ref{mpcp}(ii) to conclude condition \eqref{pmpcp2}  with
\[
p(Q)=\E_{Q_\vT}[p(Q_\theta)]=\pi_c(X_1)\cdot\E_{Q_\vT}[\theta]=\pi_c(X_1)\cdot\frac{M^\prime_\vT(r)}{M_\vT(r)} 
\]
for all $r$ in a neighbourhood of $0$ in $E_\vT$, where the last equality follows by condition \eqref{eq15a}.  	 \smallskip
	
\noindent {\bf (b)} Again by  Theorem \ref{class}, the probability measure $Q$ is a $2$-martingale measure for the process $V(\vT)$ with $ V_t(\vT)=S_t-t\cdot \vT\cdot\pi_c(X_1)$ for any $t\geq 0$, and for any $\theta\notin \wt L_{\ast\ast}$ the probability measure  $Q_\theta$ is a $2$-martingale measure for the process $V(\theta)$ with $V_t(\theta)=S_t-t\cdot \theta\cdot\pi_c(X_1)$ 	for any $t\geq 0$. In particular, for any $T>0$, Theorem \ref{ftap} 	asserts that both processes $V_{\mathbb T}(\vT)$ and $V_{\mathbb T}(\theta)$ satisfy condition (NFLVR).\smallskip
	
\noindent  {\bf (c)} Consider the reserve process $R^u(\vT):=u-V(\vT)$ ($u>0$). The equality $c(\vT)=p(Q,\vT)$, together with (a), implies that condition  \eqref{scnpc} is valid and that $Q\in\hyperlink{mslr}{\M^{ruin,2}_{S,{\bf Exp}(\rho(\vT))}}$; hence by Theorem \ref{ruin1a}, we get that $Q$ is a $2$-martingale measure for the reserve process $R^u(\vT)$,    ruin occurs $Q$-a.s. and  condition \eqref{asq3} holds true. Condition \eqref{asq3}, along with $S_{T_u(\vT)}^{(\gamma)}=-N_{T_u(\vT)}\cdot\ln c+(\frac{1}{c}-1)\cdot\sum_{j=1}^{N_{T_u(\vT)}}\ln \wb F_{X_1}(X_j)$, $\prod_{j=1}^{N_{T_u(\vT)}}\frac{d {\bf K}(\vT)}{d{\bf Exp}(\Ttheta)}(W_j)=1$ and standard computations, yields  
\[
\psi(u)=M_\vT(r)\cdot \E_Q\big[e^{-r\cdot\vT+N_{T_u(\vT)}\cdot\ln c-(\frac{1}{c}-1)\cdot\sum_{j=1}^{N_{T_u(\vT)}}\ln \wb{F}_{X_1}(X_j)}\big].
\]

Moreover, again by Theorem \ref{ruin1a}, for any $\theta\notin\wt{M}_{\ast,Q}$ the reserve processes $R^u_{\mathbb T}(\vT):=\{R^u_t(\vT)\}_{t\in\mathbb T}$ and  $r^u_{\mathbb T}(\theta):=\{r^u_t(\theta)\}_{t\in\mathbb T}$ both satisfy the (NFLVR) property.\smallskip

 In particular, if $c=2$ and $P_{X_1}=\textbf{Par}(2,5)$, then the latter formula for the ruin probability becomes
\[
\psi(u)=M_\vT(r)\cdot \E_Q\Big[e^{-r\cdot\vT}\cdot\prod_{j=1}^{N_{T_u(\vT)}}\frac{2^{\frac{7}{2}}}{(2+X_j)^{5/2}}\Big].
\] 

\end{ex}

\begin{ex}\label{exruin}
	\normalfont
Let $D:=(0,\infty)$ and  $P\in\hyperlink{msla}{\M^{\ast,2}_{S,{\bf Exp}(1/\vT)}}$.  Put $E_{X_1}:=\{r\geq{0}: M_{X_1}(r)<\infty\}$ and  define the real-valued function $\be:=\ga+\al$ on $(0,\infty)^2$ by means of  $\ga(x):=r\cdot x-\ln{M}_{X_1}(r)$ for  $r\in{E}_{X_1}$ and  for any $x>0$, and $\al(\theta):=0$ for $\theta>0$. By standard computations, get $\E_P[e^{\ga(X_1)}]=1$ and
\[
\E_P[X^2_1\cdot e^{\ga(X_1)}]=\frac{\E_{P}[X_1^2\cdot e^{r\cdot X_1}]}{\E_{P}[e^{r\cdot X_1}]}=\frac{M_{X_1}^{\prime\prime}(r)}{M_{X_1}(r)}<\infty,	
\]
since $M_{X_1}^{\prime\prime}(r)<\infty$ for all $r$ in a neighbourhood of $0$  in $E_{X_1}$; hence   $\be\in\hyperlink{fpt}{\F^{2}_{P,\vT}}$. \smallskip
	
Define the function  $\xi\in\hyperlink{mkd}{\mathfrak M_+(D)}$ by means of $\xi(\theta):=\frac{e^{-r\cdot\theta}}{\E_{P}[e^{-r\cdot\vT}]}$ for $r\in{E}_{X_1}$ and for any $\theta>0$.     Clearly $\E_P[\xi(\vT)]=1$, implying $\xi\in\hyperlink{r+d}{\mathcal R_+(D)}$, and  
\[
\E_P\left[\left(\frac{1}{\vT}\right)^2\cdot\xi(\vT)\right]=\frac{\E_P\left[\frac{1}{\vT^2} \cdot e^{-r\cdot\vT}\right]}{\E_{P}[e^{-r\cdot\vT}]}\leq \frac{\E_P\left[\frac{1}{\vT^2} \right]}{\E_{P}[e^{-r\cdot\vT}]}<\infty,
\]
where the last inequality follows  by $P\in\hyperlink{msla}{\M^{\ast,2}_{S,{\bf Exp}(1/\vT)}}$; hence  $\xi\in\hyperlink{r+da}{\mathcal R^{\ast,2}(D)}$.\smallskip

\noindent \textbf{(a)} Since $(\be,\xi)\in\hyperlink{fpt}{\F^2_{P,\vT}}\times\hyperlink{r+da}{\mathcal R_+^{\ast,2}(D)}$, applying  Theorem \ref{class}  we obtain a unique pair $(\rho,Q)\in\hyperlink{mkd}{\mathfrak M_+(D)}\times \hyperlink{msla}{\M^{*,2}_{S,{\bf Exp}(\Ttheta)}}$ determined by conditions \eqref{ast} and \eqref{martPP}, such that conditions  \eqref{rndx}  and \eqref{rnd} are valid, an essentially unique rcp $\{Q_\theta\}_{\theta>0}$ of $Q$ over $Q_\vT$ consistent with $\vT$, and a $P_\vT$-null set $\wt L_{\ast\ast}\in\B(0,\infty)$ such that for any  $\theta\notin \wt L_{\ast\ast}$ conditions $Q_\theta\in\hyperlink{mstla}{{\M}^{\ast,2}_{S,{\bf Exp}(\ttheta)}}$, \eqref{rndx}, \eqref{*} and \eqref{rcp3} hold true.  It then follows by It then follows by conditions \eqref{ast} and \eqref{rndx} that $\rho(\vT)=\frac{1}{\vT}$ $P{\uph}\sigma(\vT)$-a.s. and 
\[
Q_{X_1}(A)=\E_P[\1_{X^{-1}_1 [A]}\cdot e^{\ga(X_1)}]=\frac{\E_P[\1_{X^{-1}_1 [A]}\cdot e^{r\cdot X_1}]}{M_{X_1}(r)}\quad\text{for }  r\in{E}_{X_1}, 
\]
for any $A\in\B(0,\infty)$, respectively, while by condition \eqref{rnd}  for any $B\in\B(0,\infty)$  we get 
\begin{equation}\label{intxia}
Q_\vT(B)=\E_P[\1_{\vT^{-1}[B]}\cdot\xi(\vT)]=\frac{\E_P[\1_{\vT^{-1}_1 [B]}\cdot e^{-r\cdot \vT}]}{\E_P[e^{-r\cdot \vT}]}\quad\text{for } r\in{E}_{X_1}.
\end{equation}
Thus, for any $\theta\notin \wt L_{\ast\ast}$ the probability measure $Q_\theta$ is a PCP satisfying condition
\begin{equation}\label{11a}
p(P_\theta)=\frac{\E_{P}[X_1]}{\theta} < \frac{M^\prime_{X_1}(r)}{\theta\cdot  M_{X_1}(r)}=p(Q_\theta)<\infty,
\end{equation}
for $r$ i a neighbourhood of $0$ in $E_{X_1}$. The  inequalities hold  true, since for the function $f:E_{X_1}\rightarrow\R$ defined by means of 
$f(r):=\ln{M}_{X_1}(r)$ for all $r\in{E}_{X_1}$,  we have  $f^{\prime\prime}(r)>0$ for any $r$ in a neighbourhood $[0,r_0)$ of $0$, or equivalently $f$ is strictly convex on $[0,r_0)$, or equivalently the function $f^\prime$ with $f^\prime(r)=\frac{M_{X_1}^\prime(r)}{M_{X_1}(r)}<\infty$ for $r\in[0,r_0)$ is strictly increasing; hence $\E_P[X_1]<f^\prime(r)<\infty$ for all $r\in(0,r_0)$. \smallskip

Consequently,  condition \eqref{11a}, together Remark \ref{dispd2a}, yields condition \eqref{pmpcp1}. Since for any  $\theta\notin\wt{L}_{\ast\ast}$ the functions $\theta\mapsto\xi(\theta)$ and  $\theta\mapsto p(Q_\theta)$ are monotonic of the same monotonicity, we may apply Proposition \ref{mpcp}(ii) in order to conclude condition \eqref{pmpcp2}  with
\[
p(Q)=\E_{Q_\vT}[p(Q_\theta)]=\E_{Q_\vT}\Big[\frac{1}{\theta}\Big]\cdot\frac{M_{X_1}^\prime(r)}{M_{X_1}(r)}=\frac{\E_P\left[\frac{1}{\vT} \cdot e^{-r\cdot\vT}\right]}{\E_{P}[e^{-r\cdot\vT}]}\cdot\frac{M_{X_1}^\prime(r)}{M_{X_1}(r)},
\] 
where the last equality follows by condition \eqref{intxia}.   \smallskip

\noindent {\bf (b)}  Again by Theorem \ref{class} the probability measure $Q$ is a $2$-martingale measure for the process $V(\vT)$ with $V_t(\vT)=S_t-t\cdot \frac{1}{\vT}\cdot\frac{M_{X_1}^\prime(r)}{M_{X_1}(r)}$ for any $t\geq 0$, and  for any $\theta\notin\wt L_{\ast\ast}$ the probability measure $Q_\theta$ is a $2$-martingale measure for the process $V(\theta)$ with  $V_t(\theta)=S_t-t\cdot \frac{1}{\theta}\cdot\frac{M_{X_1}^\prime(r)}{M_{X_1}(r)}$ for any  $t\geq 0$.  In particular, for any $T>0$, Theorem \ref{ftap} 	asserts that both processes $V_{\mathbb T}(\vT)$ and $V_{\mathbb T}(\theta)$ satisfy condition (NFLVR). \smallskip

\noindent {\bf (c)} Consider the reserve process $R^u(\vT):=u-V(\vT)$ ($u>0$). The equality $c(\vT)=p(Q,\vT)$, together with (a), implies that condition  \eqref{scnpc} is valid and that $Q\in\hyperlink{mslr}{\M^{ruin,2}_{S,{\bf Exp}(\rho(\vT))}}$; hence by Theorem \ref{ruin1a} we get that $Q$ is a $2$-martingale measure for the reserve process $R^u(\vT)$,  ruin occurs $Q$-a.s. and
\[
\psi(u)=\E_P[e^{-r\cdot \vT}]\cdot \E_Q\Big[e^{r\cdot R^u_{T_u(\vT)}(\vT)+r\cdot\vT -  \frac{r\cdot{T_u(\vT)}}{\vT}\cdot\frac{M_{X_1}^\prime(r)}{M_{X_1}(r)}  +N_{T_u(\vT)}\cdot\ln{M}_{X_1}(r)}\Big]\cdot e^{-r\cdot u},
\]
where the latter follows by    \eqref{asq3} along with  
\[
S_{T_u(\vT)}^{(\ga)}=r\cdot S_{T_u(\vT)}-N_{T_u(\vT)}\cdot\ln\E_P[e^{r\cdot X_1}]\quad\text {and}\quad  \prod_{j=1}^{N_{T_u(\vT)}}\frac{d {\bf K}(\vT)}{d{\bf Exp}(\Ttheta)}(W_j)=1.
\]

Moreover, again by Theorem \ref{ruin1a}, for any $\theta\notin\wt{M}_{\ast,Q}$, the reserve processes $R^u_{\mathbb T}(\vT):=\{R^u_t(\vT)\}_{t\in\mathbb T}$ and  $r^u_{\mathbb T}(\theta):=\{r^u_t(\theta)\}_{t\in\mathbb T}$ both satisfy the (NFLVR) property.
 \end{ex}

\begin{ex}\label{expcp2}
	\normalfont
Assume that $D:=(0,\infty)$ and $P\in\hyperlink{msla}{\M^{\ast,2}_{S,{\bf Ga}(\vT,k)}}$ for  $k>0$, such that $P_{X_1}={\bf Exp}(\eta)$ with $\eta>0$ and $P_\vT={\bf Ga}(b_1,a)$ with  $b_1,a>0$. \smallskip
	
Consider the real-valued function $\be:=\ga+\al$ on $(0,\infty)^2$, with $\ga(x):=\ln(1-c\cdot\E_P[X_1])+c\cdot x$ for any $x>0$ with $c<\eta$ a positive constant, and  $\al(\theta):=\ln\big(\frac{\theta}{b}\cdot\E_{P_\theta}[W_1]\big)$ for any $\theta>0$, where $b<k$ is a positive constant. It can be easily seen that  $\E_P[e^{\ga(X_1)}]=1$ and $\E_P[X^2_1\cdot e^{\ga(X_1)}]=\frac{2}{(\eta-c)^2}<\infty$, implying  $\be\in\hyperlink{fpt}{\F^{2}_{P,\vT}}$. \smallskip
	
Let  $\xi\in\hyperlink{mkd}{\mathfrak M_+(D)}$ be defined by means of $\xi(\theta):=\left(\frac{b_2}{b_1}\right)^{a}\cdot e^{-(b_2-b_1)\cdot\theta}$ for any  $\theta>0$,  where  $b_2$ is a positive constant such that $b_2<b_1$.  Clearly $\E_P[\xi(\vT)]=1$, implying that $\xi\in\hyperlink{r+d}{\mathcal R_+(D)}$. Applying standard computations we get  
	\[
	\E_P\left[\xi(\vT)\cdot\left(\frac{e^{\ln(\frac{\vT}{b}\cdot\E_P[W_1\mid\vT])}}{\E_P[W_1\mid\vT]}\right)^2\right]=\frac{\E_P[\xi(\vT)\cdot \vT^2]}{b^2} <\infty,
	\]
implying that  $\xi\in\hyperlink{r+da}{\mathcal R^{\ast,2}(D)}$.\smallskip

\noindent {\bf (a)}	Since $(\be,\xi)\in\hyperlink{fpt}{\F^2_{P,\vT}}\times\hyperlink{r+da}{\mathcal R^{\ast,2}(D)}$,  by Theorem \ref{class}  there exist a unique pair $(\rho,Q)\in\hyperlink{mkd}{\mathfrak M_+(D)}\times \hyperlink{msla}{\M^{\ast,2}_{S,{\bf Exp}(\Ttheta)}}$ determined by  conditions \eqref{ast} and \eqref{martPP} so that  conditions  \eqref{rndx}  and \eqref{rnd} are valid, an essentially unique rcp $\{Q_\theta\}_{\theta>0}$ of $Q$ over $Q_\vT$ consistent with $\vT$, and a $P_\vT$-null set $\wt L_{\ast\ast}\in\B(0,\infty)$ such that for any  $\theta\notin \wt L_{\ast\ast}$ conditions $Q_\theta\in\hyperlink{mstla}{{\M}^{\ast,2}_{S,{\bf Exp}(\ttheta)}}$, \eqref{rndx}, \eqref{*} and \eqref{rcp3} hold true.  It then follows by conditions \eqref{ast} and \eqref{rndx} that $\rho(\vT)=\frac{\vT}{b}$ $P{\uph}\sigma(\vT)$-a.s. and $Q_{X_1}={\bf Exp}(\eta-c)$, respectively, while condition \eqref{rnd} gives $Q_\vT={\bf Ga}(b_2,a)$; hence for any $\theta\notin \wt L_{\ast\ast}$  the probability measure $Q_\theta$ is a PCP satisfying condition $p(P_\theta)=\frac{\theta}{k\cdot\eta}<\frac{\theta}{b\cdot(\eta-c)}=p(Q_\theta)<\infty$. Thus, we may apply Remark \ref{dispd2a}   in order to conclude condition \eqref{pmpcp1},  and since for all $\theta\notin\wt{L}_{\ast\ast}$  the functions $\theta\mapsto\xi(\theta)$ and  $\theta\mapsto p(Q_\theta)$ are monotonic of the same monotonicity, we may apply Proposition \ref{mpcp}(ii) in order to conclude  condition \eqref{pmpcp2} with
\[
p(Q)=\E_{Q_\vT}[p(Q_\theta)]=\E_{Q_\vT}\Big[\frac{\theta}{b\cdot(\eta-c)}\Big]=\frac{a}{b_2\cdot b\cdot(\eta-c)}.
\]

\noindent    {\bf (b)}  Again by Theorem \ref{class} the probability measure $Q$ is a $2$-martingale measure for the process $V(\vT)$ with $V_t(\vT)=S_t-t\cdot\frac{\vT}{b\cdot(\eta-c)}$ for any $t\geq 0$,	and for any  $\theta\notin \wt L_{\ast\ast}$ the probability measure $Q_\theta$ is a $2$-martingale measure for the process $V(\theta)$ with $	V_t(\theta)=S_t-t\cdot\frac{\theta}{b\cdot(\eta-c)}$ for any  $t\geq 0$. In particular, for any $T>0$, Theorem \ref{ftap} 	asserts that both processes $V_{\mathbb T}(\vT)$ and $V_{\mathbb T}(\theta)$ satisfy condition (NFLVR).\smallskip 
	
\noindent {\bf (c)} 
Consider the reserve process $R^u(\vT):=u-V(\vT)$ ($u>0$). The equality $c(\vT)=p(Q,\vT)$, together with (a), implies that condition  \eqref{scnpc} is valid and that $Q\in\hyperlink{mslr}{\M^{ruin,2}_{S,{\bf Exp}(\rho(\vT))}}$; hence by Theorem \ref{ruin1a} we get that $Q$ is a $2$-martingale measure for the reserve process $R^u(\vT)$,   ruin occurs $Q$-a.s. and  condition \eqref{asq3} holds true. Condition \eqref{asq3} together with $S_{T_u(\vT)}^{(\ga)}= N_{T_u(\vT)}\cdot\ln(1-c\cdot\E_P[X_1])+c\cdot S_{T_u(\vT)}$,
 \begin{align*}
\prod_{j=1}^{N_{T_u(\vT)}}\frac{d {\bf K}(\vT)}{d{\bf Exp}(\Ttheta)}(W_j)=e^{N_{T_u(\vT)}\cdot \big(k\cdot \ln\vT-\ln\Gamma(k)-\ln\Ttheta\big)-{T_u(\vT)}\cdot\big(\vT-\Ttheta\big)+(k-1)\cdot\sum_{j=1}^{N_{T_u(\vT)}}\ln W_j}
 \end{align*}
and standard computations, yields
\[
\psi(u)=\E_Q\big[C_2\big(N_{T_u(\vT)}, W,X,\vT\big)\big]\cdot e^{-c\cdot u},
\]
where
\begin{align*}
C_2\big(N_{T_u(\vT)}, W,X,\vT\big):=\left(\frac{b_1}{b_2}\right)^{a}\cdot\Big(\prod_{j=1}^{N_{T_u(\vT)}}\frac{\eta\cdot\vT^k\cdot W_j^{k-1}}{(\eta-c)\cdot\Gamma(k)\cdot \Ttheta}\Big)   \cdot  e^{-(b_1-b_2)\cdot\vT+c\cdot R^u_{T_u(\vT)}(\vT)   +\vT\cdot{T_u(\vT)}\cdot\frac{(1-b)\cdot(\eta-c)-c}{b\cdot(\eta-c)}}.
\end{align*}
  
Moreover, again by Theorem \ref{ruin1a}, for any $\theta\notin\wt M_{\ast,Q}$, the reserve processes $R^u_{\mathbb T}(\vT):=\{R^u_t(\vT)\}_{t\in\mathbb T}$ and  $r^u_{\mathbb T}(\theta):=\{r^u_t(\theta)\}_{t\in\mathbb T}$ both satisfy the (NFLVR) property.
 \end{ex}

It follows a counter-example to show that, the assumption of the same monotonicity of the functions $\theta\mapsto{p}(Q_\theta)$ and $\theta\mapsto\xi(\theta)$ for all $\theta\notin\wt L_{\ast\ast}$ is essential for the validity of the conclusion $p(P)\leq{p}(Q)$ in Proposition \ref{mpcp}.

\begin{cex}
	\label{cou}
	\normalfont
	
In the situation of Example \ref{expcp2}, replace $\xi\in\hyperlink{mkd}{\mathfrak M_+(D)}$ with the function $\wt \xi\in\hyperlink{mkd}{\mathfrak M_+(D)}$   defined by $\wt \xi(\theta):=\left(\frac{\wt b_2}{b_1}\right)^{a}\cdot e^{-(\wt b_2-b_1)\cdot\theta}$ for any $\theta>0$, where $\wt b_2$ is a constant satisfying $\wt b_2> \frac{b_1\cdot k\cdot\eta}{b\cdot(\eta-c)}$. By the same way as in Example \ref{expcp2} we get $(\be,\wt\xi)\in\hyperlink{fpt}{\F^2_{P,\vT}}\times\hyperlink{r+da}{\mathcal R^{\ast,2}(D)}$, and so we may apply Theorem \ref{class} in order to obtain a unique pair $(\rho,\wt Q)\in\hyperlink{mkd}{\mathfrak M_+(D)}\times \hyperlink{msla}{\M^{*,2}_{S,{\bf Exp}(\Ttheta)}}$ determined by conditions \eqref{ast} and \eqref{martPP} such that conditions  \eqref{rndx} and \eqref{rnd} hold,  an essentially unique rcp $\{\wt Q_\theta\}_{\theta>0}$ of $\wt Q$ over $\wt Q_\vT$ consistent with $\vT$, and a $P_\vT$-null set $\wt L_{\ast\ast}\in\B(0,\infty)$ such that for any  $\theta\notin \wt L_{\ast\ast}$ conditions $\wt{Q}_\theta\in\hyperlink{mstla}{{\M}^{2}_{S,{\bf Exp}(\ttheta)}}$, \eqref{rndx}, \eqref{*} and \eqref{rcp3} hold true.  It then follows by conditions \eqref{ast} and \eqref{rndx} that $\rho(\vT)=\frac{\vT}{b}$ $P{\uph}\s(\vT)$-$\text{a.s.}$ and $\wt Q_{X_1}={\bf Exp}(\eta-c)$, respectively, while by condition \eqref{rnd} we get $\wt Q_\vT={\bf Ga}(\wt b_2,a)$; hence for any $\theta\notin \wt L_{\ast\ast}$  the probability measure $\wt Q_\theta$ is a PCP satisfying condition $p(P_\theta)=\frac{\theta}{k\cdot\eta}<\frac{\theta}{b\cdot(\eta-c)}=p(\wt Q_\theta)<\infty$; hence condition \eqref{pmpcp1} holds by  Remark \ref{dispd2a}.  Easy computations show that  	
\[
p(P)=\int_D p(P_\theta)\, P_\vT(d\theta)=\frac{a}{b_1\cdot k\cdot\eta} 
\]
and 
\[
p(\wt{Q})=\int_D p(\wt Q_\theta)\, Q_\vT(d\theta)=\frac{a}{\wt b_2\cdot b\cdot(\eta-c)}, 
\]
implying $p(P)>p(\wt{Q})$; hence the conclusions (i) and (ii)  of Proposition \ref{mpcp} fail. Note that all assumptions of  this proposition except for that of the same monotonicity of the functions $\theta\mapsto\xi(\theta)$ and $\theta\mapsto{p}(\wt{Q}_\theta)$ for all $\theta\notin\wt L_{\ast\ast}$  are satisfied, since the function $\theta\mapsto\xi(\theta)$ is strictly decreasing, while $\theta\mapsto{p}(\wt{Q}_\theta)$ is strictly increasing. As a consequence, we infer that the assumption of the same monotonicity of the functions $\theta\mapsto\xi(\theta)$ and $\theta\mapsto{p}(Q_\theta)$ is essential for the validity of the conclusion of Proposition \ref{mpcp}.\smallskip
	
 Note that even if $p(P_\theta)=p(\wt Q_\theta)$ for any $\theta\notin \wt L_{\ast\ast}$, i.e., whenever $c=0$ and $b=k$, the equality $p(P)=p(\wt Q)$ fails.
 \end{cex}


\begin{thebibliography}{99}
\bibitem{acl} Albrecher, H., Constantinescu, C., Loisel, S. (2011). Explicit ruin formulas for models with dependence among risks. \textit{Insur. Math. Econ.}, \textbf{48}, 265--270.
	

\bibitem{asal} Asmussen, S. and Albrecher, H. (2010)  \textit{Ruin Probabilities}, 2nd edn.. World Scientific, Singapore.

\bibitem{ba} Bauer, H. (1996). \textit{Probability theory} (Vol. 23). Walter de Gruyter.
 
 
\bibitem{Co}  Cohn, D.~L. (2013).  {\em Measure Theory},  2nd edn.. Birkhäuser Advanced Texts: Basler Lehrb\"{u}cher. Birkh\"{a}user/Springer, New York.


\bibitem{csz} Constantinescu, C., Samorodnitsky, G. and Zhu, W. (2018). Ruin probabilities in classical risk models with gamma claims. \textit{Scand. Actuar. J.}, \textbf{2018}(7), 555--575.
	
\bibitem{da} Dassios, A. (1987).   \textit{Insurance, storage and point processes: an approach via piecewise deterministic Markov processes}. Ph.D. Thesis. Imperial College, London.
	
\bibitem{daem} Dassios, A. and Embrechts, P. (1989). Martingales and insurance risk. \textit{Commun. Statist. Stochastic Models}, \textbf{5}, 181--217.

\bibitem{dh}   Delbaen, F. and Haezendonck, J. (1989). A martingale approach to premium calculation principles in an arbitrage free market. {\em Insur. Math. Econ.}, {\bf 8}(4), 269--277. 
	
\bibitem{ds}  Delbaen, F. and Schachermayer, W. (2006). {\em The Mathematics of Arbitrage}. Springer-Verlag Berlin Heidelberg.
	
\bibitem{em} Embrechts, P. (2000). Actuarial versus Financial Pricing of Insurance. \textit{Journal of Risk Finance} \textbf{1}(4), 17--26.
	
\bibitem{emme}  Embrechts, P. and Meister, S. (1997).  Pricing Insurance Derivatives: The Case of CAT-Futures. In \textit{Securitization of Insurance Risk: The 1995 Bowles Symposium}, pp. 15--26. SOA Monograph M-FI97-1.

\bibitem{esg} Embrechts, P., Schmidli, H. and Grandell, J. (1993). Finite-time Lundberg inequalities in the Cox case, \textit{Scand. Actuar. J.}, \textbf{1993}(1), 17--41.

\bibitem{fr1} Fremlin, D.~H. (2011). \textit{Measure Theory} (Vol. 1).  Torres Fremlin (Ed.).
 
\bibitem{gr} Grandell, J. (1997). \textit{Mixed Poisson Processes}. Chapman \& Hall/CRC.
	
\bibitem{gut}  Gut, A. (2009).  \textit{Stopped Random Walks: Limit Theorems and Applications},  2nd edn.. Springer New York.
	
\bibitem{ja} Jacod, J. (1979) \textit{Calcul Stochastique et Probl\`{e}mes de Martingales}. Springer-Verlag Berlin Heidelberg.


	
\bibitem{hk} Harrison, J.~M. and Kreps, D.~M. (1979). Martingales and arbitrage in multiperiod securities markets. \textit{J. Econ. Theory}, {\bf 20}(3), 381--408.

\bibitem{lw} Landriault, D. and Willmot, G. (2008). On the Gerber–Shiu discounted penalty function in the Sparre Andersen model with an arbitrary interclaim time distribution. {\em Insur. Math. Econ.}, \textbf{42}(2), 600--608.

\bibitem{lm1v}  Lyberopoulos, D.~P. and Macheras, N.~D.  (2012). Some characterizations of mixed Poisson processes. \textit{Sankhy\={a} A}, {\bf 74}(1), 57--79. 
	
\bibitem{lm3} Lyberopoulos, D.~P. and Macheras, N.~D.   (2019).  A characterization of martingale-equivalent compound mixed Poisson processes. \textit{ArXiv Mathematics e-prints}. arXiv: 1905.07629.
 
\bibitem{lm3aap} Lyberopoulos, D.~P. and Macheras, N.~D.  (2021).  A characterization of martingale-equivalent mixed compound Poisson processes. \textit{Ann. Appl. Probab.}, \textbf{31}(2), 778--805.
	
\bibitem{lm6z3}  Lyberopoulos, D.~P. and Macheras, N.~D. (2022). Some characterizations  of mixed renewal processes. \textit{Math. Slovaca}, \textbf{72}(1), 197--216.
	
\bibitem{mt3} Macheras, N.~D. and Tzaninis, S.~M.  (2020). A characterization of equivalent martingale measures in a renewal risk model with applications to premium calculation principles. \textit{Mod. Stoch.: Theory Appl.}, {\bf 7}(1), 43--60.
	
\bibitem{ny} Ng, A.~C.~Y. and Yang, H. (2005).  Lundberg-type bounds for the joint distribution of surplus immediately before and at ruin under the Sparre Andersen model, \textit{NAAJ}, \textbf{9}, 85--100.	

\bibitem{pro}  Protter, P.~E. (2005). \textit{Stochastic Integration and Differential Equations}, ,  2nd edn.. Springer, Berlin Heidelberg.	
	
\bibitem{rss}   Rolski, T.,  Schmidli, H., Schmidt, V. and  Teugels, J.~L. (1999). \textit{Stochastic Processes for Insurance and Finance}. Wiley Series in Probability and Statistics. John Wiley \& Sons, Ltd., Chichester.
	
\bibitem{scm95} Schmidli, H., (1995). Cramér-Lundberg approximations for ruin probabilities of risk processes perturbed by diffusion. \textit{Insur. Math. Econ.}, \textbf{16}, 135--149.
	
\bibitem{scm10} Schmidli, H. (2010). On the Gerber–Shiu function and change of measure. \textit{Insur. Math. Econ.}, \textbf{46}(1), 3--11.
	
\bibitem{scm}  Schmidli, H. (2017). \textit{Risk Theory}.  Springer Actuarial. Springer, Cham. 
	
\bibitem{Sc}   Schmidt, K.D.  (1996). {\em Lectures on Risk Theory}. B.G. Teubner, Stuttgart.
	
\bibitem{Sc1} Schmidt, K.D. (2014). On inequalities for moments and the covariance of monotone functions.  {\em Insur. Math. Econ.}, {\bf 55}, 91--95.
	
\bibitem{so}  Sondermann, D. (1991). Reinsurance in arbitrage-free markets.  {\em Insur. Math. Econ.}, {\bf 10}(3), 191--202.
	
\bibitem{t1} Tzaninis, S. M. (2022). Applications of a change of measures technique for compound mixed renewal processes to the ruin problem. \textit{Mod. Stoch. Theory Appl.}, \textbf{9}(1), 45--64.

\bibitem{mt2}  Tzaninis, S.~M. and Macheras, N.~D. (2023). A characterization of progressively equivalent probability measures preserving the structure of a compound mixed renewal process. \textit{ALEA, Lat. Am. J. Probab. Math. Stat.}, \textbf{20}, 225--247.
	
\bibitem{vw}   von Weizs\"{a}cker, H. and Winkler, G. (1990). {\em Stochastic Integrals: An Introduction}. Vieweg+Teubner Verlag Wiesbaden.
	
\bibitem{wangs} Wang, S. (1995). Insurance pricing and increased limits ratemaking by proportional hazards transforms. \textit{Insur. Math. Econ.}, {\bf 17}(1), 43--54. 
 
\end{thebibliography}
\end{document}